\documentclass{article}

\usepackage[utf8]{inputenc}
\usepackage[margin=2cm]{geometry}
\usepackage{fullpage,enumitem,amssymb,amsmath,amsthm,xcolor,cancel,graphicx,tikz,tikz-cd, mathtools, subcaption}
\usepackage{indentfirst}

\usepackage[breaklinks,colorlinks,plainpages,hypertexnames=false,plainpages=false]{hyperref}
\hypersetup{urlcolor=blue, citecolor=blue, linkcolor=blue}
\usepackage[capitalise, noabbrev]{cleveref} 

\usepackage{nicematrix}
\usepackage{xr}
\usepackage{environ, thmtools}
\usepackage[normalem]{ulem}
\usetikzlibrary{calc}

\newtheorem{theorem}{Theorem}[section]
\newtheorem{lemma}[theorem]{Lemma}

\newtheorem{proposition}[theorem]{Proposition}
\newtheorem{corollary}[theorem]{Corollary}

\theoremstyle{definition}
\newtheorem{definition}[theorem]{Definition}
\newtheorem{example}[theorem]{Example}
\AtBeginEnvironment{example}{%
	\pushQED{\qed}%
}
\AtEndEnvironment{example}{\popQED\endexample}

\newtheorem{remark}[theorem]{Remark}

\newcommand{\R}{\mathbb{R}}
\newcommand{\Z}{\mathbb{Z}}
\newcommand{\ass}{\delta < \half \min(\Xi(V), \Xi(W))}

\newcommand{\Stab}[1]{\text{Stab}(#1)}
\newcommand{\Barc}[1]{\mathbf{Bar}(#1)}

\newcommand{\qhalf}{\frac{q}{2}}
\newcommand{\half}{\frac{1}{2}}
\newcommand{\q}{{\geq q}}

\newcommand{\pperm}{\begin{psmallmatrix} 0 & 1 \\ 1 & 0 \end{psmallmatrix}}
\newcommand{\pprojfirst}{\begin{psmallmatrix} 1 & 0 \end{psmallmatrix}}
\newcommand{\pprojsecond}{\begin{psmallmatrix} 0 & 1 \end{psmallmatrix}}
\newcommand{\pinjfirst}{\begin{psmallmatrix} 1 \\ 0 \end{psmallmatrix}}

\newcommand{\vfirst}{\begin{psmallmatrix} 1 & 0 \\ 0 & 1 \\ 0 & 0  \end{psmallmatrix}}
\newcommand{\vsecond}{\begin{psmallmatrix} 0 & 1 & 0 \end{psmallmatrix}}

\newcommand{\todot}{\mathrel{\ooalign{\hfil$\vcenter{
				\hbox{$\scriptscriptstyle\bullet$}}$\hfil\cr$\to$\cr}}}
\newcommand{\mapstodot}{\mathrel{\ooalign{\hfil$\vcenter{
				\hbox{$\scriptscriptstyle\bullet$}}$\hfil\cr$\mapsto$\cr}}}

\newcommand{\xrightarrowdbl}[2][]{\xrightarrow[#1]{#2}\mathrel{\mkern-14mu}\rightarrow}

\DeclareMathOperator\vs{Vec}
\DeclareMathOperator\supp{supp}
\begingroup
\catcode`\&=13
\gdef\pampmatrix{%
	\begingroup
	\let&=\amsamp
	\begin{pmatrix}%
	}
	\gdef\endpampmatrix{\end{pmatrix}\endgroup}
\endgroup

\makeatletter
\def\namedlabel#1#2{\begingroup
	#2 \def\@currentlabel{#2} \phantomsection\label{#1}\endgroup}
\makeatother

\makeatletter
\newcommand\IfRestateTF{%
	\ifx\label\thmt@gobble@label 
	\expandafter\@firstoftwo
	\else
	\expandafter\@secondoftwo
	\fi
}
\makeatother

\title{Ladder Decomposition for Morphisms of Persistence Modules}
\author{Živa Urbančič \\ Durham University \\ \texttt{ziva.urbancic@durham.ac.uk} \and Jeffrey Giansiracusa \\ Durham University \\ \texttt{jeffrey.giansiracusa@durham.ac.uk}}
\date{\today}

\begin{document}
	\maketitle
	
	\begin{abstract}
		The output of persistent homology is an algebraic object called a persistence module. This object admits a decomposition into a direct sum of interval persistence modules described entirely by the barcode invariant. In this paper we investigate when a morphism~$\Phi \colon V \to W$ of persistence modules admits an analogous direct sum decomposition. Jacquard et al.~\cite{jacquard2021space} showed that a ladder decomposition can be obtained whenever the barcodes of $V$ and $W$ do not have any strictly nested bars. We refine this result and show that even in the presence of nested bars, a ladder decomposition exists when the morphism is sufficiently close to being invertible relative to the scale of the nested bars.
	\end{abstract}
	
	\section{Introduction} \label{sec:intro}
	The main aim of the field of topological data analysis is to develop methods that detect properties related to the shape of the problem of study from a data set of observations. The introduction of persistent homology~\cite{edelsbrunner2000topological} is oftentimes considered as the birth of the field and is still one of its most commonly applied methods. To compute it, the underlying topology of a data set is encoded in the form of a nested family of objects (most often simplicial complexes or sets) reflecting the structure detected at different scales. Computing homology groups with coefficients in field~$\mathbb{F}$ for each of these objects enhanced with the morphisms induced by inclusions gives a \emph{persistence module}: a family~$V=\{V_i\}$ of vector spaces indexed over a poset~$P$ with each relation~$i \leq j$ in~$P$ giving a linear map~$v_{i, j} \colon V_i \to V_j$ called an \emph{inner morphism}. A simple example is an~\emph{interval persistence module}~$V=k_I$ for some interval~$I$ in~$P$, which is given by~$V_i = \mathbb{F}$ for~$i \in I$ and~$V_i = 0$ for~$i \notin I$ with the accompanying maps~$v_{i, j}$ being the identity whenever~$i, j \in I$ and the zero map otherwise. One-parameter \emph{point-wise finite dimensional} (or p.f.d.) persistence modules, which are indexed by a finite totally ordered set and each of the vector spaces~$V_i$ is finite dimensional, are especially nice to work with. The \emph{structure theorem}~\cite{crawley2015decomposition} states that any p.f.d.\ persistence module~$V$ is a direct sum of interval persistence modules. The multiset of intervals appearing in the decomposition is a topological invariant called the~\emph{persistence barcode}~\cite{carlsson2005persistence} which we denote as~$\Barc{V}$. It is complete and its discrete nature makes it easy to understand and visualise.
	
	A morphism~$\Phi \colon V \to W$ of persistence modules is a family~$\{\Phi_i \colon V_i \to W_i \}_{i \in P}$ of linear maps that commute with the inner morphisms. When~$P$ is finite and totally ordered, the morphism can be represented as the commutative diagram
	\begin{center}
		\begin{tikzcd}
			V_0 \arrow[r] \arrow[d, "{\Phi_0}"] & V_1 \arrow[r] \arrow[d, "{\Phi_1}"] & \cdots \arrow[r] & V_{l-1} \arrow[r]  \arrow[d, "{\Phi_{l-1}}"] & V_l \arrow[d, "{\Phi_l}"] \\
			W_0 \arrow[r]                       & W_1 \arrow[r]                       & \cdots \arrow[r] & W_{l-1} \arrow[r]                            & W_l.
		\end{tikzcd}
	\end{center}
	A common approach is to view it as being a persistence module itself: the indexing poset is ``ladder-like''(see~\cref{fig:ladder_poset}) and so it belongs among \emph{ladder persistence modules}~\cite{escolar2016ladder_modules}.
	\begin{figure}[ht].
		\centering
		\includegraphics[width=0.5\linewidth]{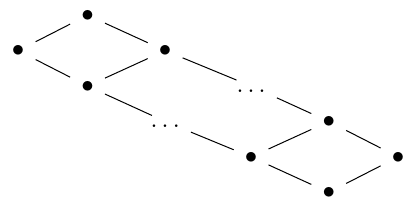}
		\caption{Hasse diagram of a finite ``ladder-like'' poset.}
		\label{fig:ladder_poset}
	\end{figure}
	In the spirit of the structure theorem giving a barcode for one-parameter p.f.d.\ persistence modules, Jacquard et al.~\cite{jacquard2021space} identify assumptions under which a morphism viewed as a ladder persistence module decomposes into a direct sum of elementary ladder persistence modules:
	\begin{itemize}
		\item $\mathbf{I}^+_J$ consists of an interval~$J$ on the source side that is mapped to~$0$ on the target side.
		\item $\mathbf{I}^-_K$ consists of~$0$ on the source side and an interval~$K$ on the target side.
		\item $\mathbf{R}^J_K$ consist of an interval~$J$ on the source side and an interval~$K$ on the target side with a morphism that sends the generator of~$J$ to the generator of~$K$.
	\end{itemize}
	A ladder decomposition is then an isomorphism of~$\Phi \colon V \to W$ to a direct sum of elementary ladder persistence modules, and the main result of~\cite{jacquard2021space} is that it exists whenever the barcodes~$\Barc{V}$ and~$\Barc{W}$ do not admit strictly nested bars.

	In this work we focus on a special family of morphisms called the \emph{interleaving morphisms}. They are associated with a shifting parameter~$\delta$ and come in pairs~$(\Phi, \Psi)$, with~$\Phi\colon V \to W(\delta)$ and~$\Psi\colon W \to V(\delta)$, where the persistence module~$V(\delta)$ is obtained from~$V$ via shifts in the indexing parameter, i.e.~$V(\delta)_t = V_{t+\delta}$. The pair must further satisfy that their composition (both instances) is exactly the family of inner morphisms of the corresponding persistence module. Because of this property, their existence implies the persistence modules in the domain and codomain are algebraically related, and the smaller the shifting parameter, the stronger this relation is. The smallest~$\delta$ for which a~$\delta$-interleaving pair between~$V$ and~$W$ exists is denoted by~$d_I$ and called their~\emph{interleaving distance}~\cite{interleavings_original_paper} (it can be defined in more general settings, for example in any category with flow~\cite{de2018theory}). As a consequence of various stability results~\cite{bottleneck_distance, bauer2013induced} interleaving morphisms arise between persistence modules obtained from closely related inputs to the persistent homology pipeline. They also arise between restrictions of multi-parameter persistence modules to close-enough parallel lines in the parameter space~\cite{fibered_barcode}. As such, they are interesting from the point of view of both applications and theoretical results.
	
	In most of our work we will only be interested in one of the morphisms in the interleaving pair. For this reason we introduce the notion of a~\emph{$\delta$-invertible morphism} - a morphism~$\Phi \colon V \to W$ for which there exists another morphism~$\Psi: W \to V(2\delta)$ so that both of their compositions are just the inner morphisms in the corresponding persistence module. It is easy to see a~$\delta$-invertible morphism~$\Phi \colon V \to W$ is equivalent to a morphism in a~$\delta$-interleaving pair between~$V$ and~$W(-\delta)$. Special properties of~$\delta$-invertible morphisms make them somewhat easier to work with than the general morphisms, which we leverage to obtain their ladder decompositions under looser assumptions. To elaborate, we define a constant~$\Xi(V)$ called the~\emph{nestedness} of persistence module~$V$, as the minimal distance between endpoints (either of birth-points or of death-points) of strictly nested bars.
	\begin{restatable*}{theorem}{main} \label{thm:Xi_constant}
	For a~$\delta$-invertible morphism~${\Phi\colon V \to W}$ with~${\delta < \half \min( \Xi(V), \Xi(W))}$ there exist parameters~${r_{J_1}^{J_2}, d_{J}^{+}, d_K^- \in \mathbb{N}}$ such that 
	\begin{align*}
		(V, W, \Phi) \cong \bigoplus_{J_1 \preceq J_2}\Big( \mathbf{R}_{J_1}^{J_2} \Big)^{r_{J_1}^{J_2}} \oplus \bigoplus_{J} \Big( \mathbf{I}^+_{J} \Big)^{d_{J}^+} \oplus \bigoplus_{K} \Big( \mathbf{I}^-_{K} \Big)^{d_{K}^-}.
	\end{align*}
	\end{restatable*}
	In persistent homology pipelines, long bars are the signals and short bars are often associated with noise in the input. Therefore it is potentially useful to truncate and discard the bars shorter than some threshold. Here we explore how this process interacts with the ladder decomposition theorem above. For this reason, we introduce another parameter~$q$ to be the length of the longest bar we wish to disregard. We consider restrictions~$V_\q$ and~$W_\q$ of modules~$V$ and~$W$ to the features which persist for at least~$q$, which are given by projection maps~$pr_\q^V$ and~$pr_\q^W$ and inclusion maps~$i_\q^V$ and~$i_\q^W$ respectively.
	\begin{restatable*}{theorem}{qmain} \label{prop:q-decompositions_invertible}
		Let~$\Phi\colon V \to W$ be a~$\delta$-invertible morphism. If there exists a parameter~$q$ such that
		\begin{enumerate}
			\item $\delta < \half \min(\Xi(V), \Xi(W_{\geq q})) - \qhalf$, then the $(\delta+\qhalf)$-invertible morphism~$pr_\q^W \circ \Phi \colon V \to W_{\geq q}$ decomposes as a ladder persistence module. \label{invertible1}
			\item $\delta < \half \min(\Xi(V_{\geq q}), \Xi(W)) - \qhalf$, then the $(\delta+\qhalf)$-invertible morphism $\Phi \circ i_\q^V \colon V_{\geq q} \to W$ decomposes as a ladder persistence module. \label{invertible2}
			\item $\delta < \half \min(\Xi(V_{\geq q}), \Xi(W_{\geq q})) - \qhalf$, then the $(\delta+\qhalf)$-invertible morphism $pr_\q^W \circ \Phi \circ i_\q^V \colon V_{\geq q} \to W_{\geq q}$ decomposes as a ladder persistence module. \label{invertible3}
		\end{enumerate}
		\IfRestateTF{}{Furthermore, the barcode bases in which these ladder decompositions are obtained can be extended to barcode bases of persistence modules~$V$ and~$W$.}
	\end{restatable*}

	Via the correspondence between~$\delta$-invertible morphisms and morphisms appearing in~$\delta$-interleaving pairs we obtain ladder decompositions of both morphisms in an interleaving pair. Comparing them we find they are as compatible as the shifting parameter allows.
	\begin{restatable*}{theorem}{correspondence} \label{thm:similarity_of_bases}
		Let~$(\Phi,\Psi)$ be a~$\delta$-interleaving pair between modules~$V$ and~$W$ with~$\ass$. For any pair of bars~$J_V \in \Barc{V}$ and~$J_W \in \Barc{W}$ satisfying~${|J_V|, |J_W| \geq 2\delta}$, and for any~$\mu \in \mathbb{N}$ the following statements are equivalent
		\begin{itemize}
			\item $(\mathbf{R}_{J_W(\delta)}^{J_V})^\mu$ appears in the ladder decomposition of~$\Phi$,
			\item $(\mathbf{R}_{J_V(\delta)}^{J_W})^\mu$ appears in the ladder decomposition of~$\Psi$.
		\end{itemize}
	\end{restatable*}
	As observed already in~\cite{jacquard2021space}, whenever a ladder decomposition can be obtained, it induces a \emph{partial matching} on the barcodes of the persistence modules involved. A partial matching can be defined on a general multi-set as a partial bijection. The first instance of a morphism-induced partial matching appeared in~\cite{bauer2013induced} in order to prove the~\emph{Isometry theorem}, stating that the interleaving distance on the one-parameter persistence modules agrees with the \emph{bottleneck distance}~\cite{bottleneck_distance} on persistence barcodes. This construction, however, is not linear with respect to direct sums of ladder persistence modules. Addressing this (and some other grievances) the notion of \emph{basis-independent} partial matchings has been introduced~\cite{basis_independent_matchings}, which the ladder decomposition induced partial matchings are examples of. We analyse their properties when the morphism is a part of an interleaving and show that their cost is limited above by the interleaving parameter.
	\begin{restatable*}[of \cref{thm:Xi_constant}]{corollary}{cost} \label{cor:cost_of_induced_matching}
		Let~$\Phi \colon V \to W(\delta)$ be one of two morphisms making a~$\delta$-interleaving pair for~$\delta < \half \min( \Xi(V), \Xi(W))$, and~$\chi_\Phi$ the partial matching induced by the ladder decomposition of~$\Phi$. Its cost is at most~$\delta$. If further~$\delta = d_I(V, W)$, then the induced matching realizes the bottleneck distance.
	\end{restatable*}
	Further, the matchings induced by ladder decompositions of the pair of morphisms making an interleaving are compatible for all bars of sufficient length.
	\begin{restatable*}[of \cref{thm:similarity_of_bases}]{corollary}{matchings} \label{cor:correspondence_of_matchings}
		Let~$\chi_\Phi \colon \Barc{V} \todot \Barc{W}$ and~$\chi_{\Psi} \colon \Barc{W} \todot \Barc{V}$ be the partial matchings induced by morphisms~$(\Phi, \Psi)$ forming a $\delta$-interleaving pair where~$\ass$. For any pair of bars~$J_V \in \Barc{V}$ and~$J_W \in \Barc{W}$ satisfying~$|J_V| \geq 2\delta$ and~$|J_W| \geq 2\delta$, and any~$\mu \in \mathbb{N}$
		\begin{align*}
			((J_V, J_W), \mu) \in \chi_\Phi \iff ((J_W, J_V), \mu) \in \chi_\Psi.
		\end{align*} 
	\end{restatable*}
	
	\begin{center}
		\textbf{\large{Structure of the Paper}}
	\end{center}
	The notation used throughout the paper is introduced in \cref{sec:persistence_modules_and_barcode_bases}, which also includes a summary of the theory of \emph{barcode bases}~\cite{jacquard2021space}. Further, we specify how to view morphisms as ladder persistence modules and state the \emph{ladder decomposition theorem} of~\cite{jacquard2021space}. We dedicate \cref{sec:morphisms_and_induced_matchings} to the interleavings and~$\delta$-invertible morphisms. We quote their definition and state some basic results in the language of barcode bases in \cref{subsec:delta-interleavings}. We define the~\emph{nestedness} constant and state the main theorem about the ladder decomposition of~$\delta$-invertible morphisms in \cref{subsec:decomp_of_interleavings}. It also contains a plethora of technical lemmata used to prove the main theorem. \cref{subsec:ladder_decomp_of_interleaving_pair} compares the ladder decompositions of both morphisms making an interleaving pair. The generalisation of the theory introduced in \cref{sec:morphisms_and_induced_matchings} to the case when we discard short bars is presented in \cref{sec:quasi-barcode_form}. It includes the definition of a~$q$-\emph{splitting} of a persistence module and the generalisation of the main theorem -- \cref{prop:q-decompositions}. Lastly, we state results regarding the ladder decomposition induced partial matchings in \cref{sec:induced_partial_matchings}, where they are also compared with other notions of induced partial matchings.
		
	\begin{center}
		\textbf{\large{Acknowledgments}}
	\end{center}
	The authors are a part of the Centre for Topological Data Analysis supported by the EPSRC grant New Approaches to Data Science: Application Driven Topological Data Analysis with reference EP/R018472/1.
	
	\section{Preliminaries}
\label{sec:persistence_modules_and_barcode_bases}

Let us introduce some preliminary definitions and notation used throughout the paper in this section. Particular attention is given to the notion of \emph{barcode bases} originally introduced in~\cite{jacquard2021space}. For a thorough expos\'{e} on the theory of persistence, consult~\cite{edelsbrunner2022computational, carlsson2009topology}.

First, define the following relations on the set of intervals:
	\begin{align*}
		&[i_1, j_1] \leq [i_2, j_2] \iff i_1 < i_2 \text{ or } (i_1 = i_2 \text{ and }j_1 < j_2), \\
		&[i_1, j_1] \preceq [i_2, j_2] \iff i_1 \leq i_2 \leq j_1 \leq j_2 \\
		&[i_1, j_1] \subset [i_2, j_2] \iff i_2 < i_1 \leq j_1 < j_2
	\end{align*}
The first,~$\leq$, is simply the lexicographical order (total), while the second,~$\preceq$, is not transitive and therefore not an order. The relation~$\preceq$ is uniquely useful in persistence barcodes, since it encodes the information of when the generator of one bar can be mapped to the generator of the other with a morphism of persistence modules (see \cref{rem:how_mphi_can_map}). It is often referred to as the \emph{overlapping relation}~\cite{bauer2020persistence}. Finally, the relation~$I \subset J$ simply states that~$I$ is \emph{strictly nested} in~$J$. Note that when a pair of bars~$I \leq J$ has a nonempty intersection and~$I \npreceq J$, it must be nested as~$J \subset I$.

\subsection{Persistence Modules}
Persistent modules arise in the study of topological properties of filtered spaces when we apply homology to each step in the filtration. Formally, a persistence module is a covariant functor~$V \colon P \to \vs$, where~$P$ is a totally ordered set of the form~$[l+1] = \{0, 1, \ldots, l\}$ for some~$l \in \mathbb{N}$ viewed as a category, and~$\vs$ is the category of vector spaces. In other words,~$V$ assigns a vector space~$V_p$ to each element~$p \in P$ and a linear map~$v_{p, q} \colon V_p \to V_q$ to any pair of indices~$p, q \in P$ with~$p \leq q$. Throughout this paper we also assume that the vector space~$V_p$ is finite dimensional for each~$p \in P$ (such persistence modules are pointwise finite dimensional or p.f.d. for short).
\begin{remark}\label{r:poset_vs_toset}
	Persistence modules can be indexed by more general sets. In particular, \cref{sec:morphisms_and_induced_matchings,sec:quasi-barcode_form} will consider persistence modules indexed by ``ladder-like'' posets.
\end{remark}
These assumptions are sensible, since it is the normal setting for the analysis of one-parameter filtrations built on finite real-world data sets, in which the homology changes finitely many times when we vary the filtration parameter. Under these assumptions a persistence module over field~$\mathbb{F}$ can be represented with a diagram
	\begin{align*}
		V_0 \ {\xlongrightarrow{\mathmakebox[1cm]{v_{0, 1}}}}\ V_1\ {\xlongrightarrow{\mathmakebox[1cm]{v_{1,2}}}}\ \cdots\ {\xlongrightarrow{\mathmakebox[1cm]{v_{l-2,l-1}}}}\ V_{l-1}\ {\xlongrightarrow{\mathmakebox[1cm]{v_{l-1,l}}}}\ V_l.
	\end{align*}
Further, the structure theorem~\cite{crawley2015decomposition} assures there exists an interval decomposition
\begin{align} \label{eq:interval_decomposition}
	V \cong \bigoplus_{J} k_J,
\end{align}
where~$k_J$ is the interval persistence module of interval~$J$ defined to be~$\mathbb{F}$ for any index~$j \in J$, zero for all other indices and the inner morphisms being the identity whenever possible. The multiset
$$\Barc{V} = \{(J, \mu_J) \mid k_J \text{ appears in the interval decomposition } \mu_J \text{-many times}\}$$
is called the \emph{barcode} of persistence module~$V$.

Set $n_i = \text{dim}_{\mathbb{F}}V_i$ for all $i \in [l+1]$ and select a \emph{basis family}
	$$\mathcal{B} = \{B_i \subset V_i \ | \ B_i \text{ is an ordered basis of } V_i \text{ for all } i \in [l+1] \}.$$
In these bases the inner morphisms~$v_{i-1, i}\colon V_{i-1} \to V_i$ can be represented as $n_i \times n_{i-1}$ matrices~$A_i$. Consequently, the module~$V$ is isomorphic to
	\begin{align*}
		\mathbb{F}^{n_0} \ {\xrightarrow{\mathmakebox[0.7cm]{A_1}}}\ \mathbb{F}^{n_1}\ {\xrightarrow{\mathmakebox[0.7cm]{A_2}}}\ \cdots\ {\xrightarrow{\mathmakebox[0.7cm]{A_{l-1}}}}\ \mathbb{F}^{n_{l-1}}\ {\xrightarrow{\mathmakebox[0.7cm]{A_l}}}\ \mathbb{F}^{n_l}.
	\end{align*}

Intuitively, a barcode basis is a basis family~$\mathcal{B}$ in which the inner morphisms send basis vectors to basis vectors or the zero vector, and no two basis vectors ever have the same non-zero image. Thus a bar corresponds to a sequence of basis vectors, each one mapping to the next. To define it explicitly use the fact that matrices~$A$ written in such a basis family take a special form.
\begin{definition}\label{def:barcode_form}
	An $m\times n$ matrix~$A=(A_{ij})$ of rank~$r$ is in \emph{barcode form} if there exists a strictly increasing function~$c:\{1, 2, \ldots, r\} \to \{1, 2, \ldots, n \}$ so that
	\begin{align*}
		A_{ij} = \begin{cases}
			1, & \text{if } j = c(i),\\
			0, & \text{otherwise.}
		\end{cases}
	\end{align*}
\end{definition}
\begin{example} \label{ex:barcode_shape}
	A simple example of a matrix in barcode form is the identity matrix. It is of full rank and the function~$c$ from the definition is the identity. A general example, however, is a matrix in row-echelon form where the pivots are~$1$ and all the other entries are~$0$, as 
		\[
	\left[\;\begin{NiceArray}{ccccc}[name=B]
		1 & 0 & 0 & 0 & 0 \\
		0 & 1& 0 & 0 & 0\\
		0 & 0& 0 & 1 & 0\\
		0 & 0& 0 & 0 & 0\\
	\end{NiceArray}\;\right].
	\begin{tikzpicture}[remember picture, overlay]
		\draw[thin,gray] let \p1=($(B-1-2.south west)-(B-2-1.north east)$) in
		([xshift=-\x1/2,yshift=-2*\y1/3]B-1-1.south west) 
		-| ([xshift=-\x1/2,yshift=-2*\y1/3]B-2-2.south west)
		-| ([xshift=-\x1/2,yshift=-2*\y1/3]B-3-4.south west)
		-- ([xshift=\x1/3,yshift=-2*\y1/3]B-3-5.south east);
	\end{tikzpicture} \qedhere
	\]
\end{example}

\begin{definition}
	A basis family~$\mathcal{B}$ is a \emph{barcode basis} if all matrices~$A_i$ are in barcode form.
\end{definition}

Let~$G = \prod_{i=0}^l \text{GL}(n_i; \mathbb{F})$ be the group of elements~$g = (g_0, g_1, \ldots, g_l)$ which acts on a matrix sequence~${A = \{A_i\}_i}$ via a change of basis. More precisely,~$(gA)$ is given by
\begin{align*}
	(gA)_i = g_i \cdot A_i \cdot g_{i-1}^{-1},
\end{align*}
which corresponds to switching from basis family~${\mathcal{B} = \{ B_i \}}$ to~$g\mathcal{B} = \{ g_i B_i \}$. By~\cite[Proposition 2.4.]{jacquard2021space} any matrix sequence~$A \in X$ can be put in a barcode form by the action of~$G$.

The choice of barcode basis for a persistence module is not unique, which becomes obvious when considering persistence modules whose barcodes include a bar with multiplicity higher than~$1$. To discern the space of barcode bases authors of~\cite{jacquard2021space} examine the changes of basis that keep the matrix sequence~$A$ unchanged. These changes are part of the \emph{stabiliser} of~$A$ given by
\begin{align} \label{eq:stabiliser}
	\Stab{A} = \{ g \in G \ | \ A_i = g_i \cdot A_i \cdot g_{i-1}^{-1} \text{ for all } 1 \leq i \leq l \}.
\end{align}
The space of barcode bases is then the orbit~$\Stab{A} \mathcal{B}$, where~$\mathcal{B}$ is a barcode basis.

\subsection{Morphism between Persistence Modules} \label{subsec:mphi_as_lpm}
Here we recall some facts about the algebraic structure of morphisms of persistence modules, leading up to the structure theorem for ladder persistence modules of~\cite{jacquard2021space}.

Given persistence modules~$V$ and~$W$ indexed by an ordered set~$P$, a morphism~$\Phi \colon V \to W$ consists of a family~$\{\Phi_i \colon V_i \to W_i\}_{i \in P}$ of linear maps satisfying
\begin{align*}
	\Phi_{i+1} \circ v_{i, i+1} = w_{i,i+1} \circ \Phi_i.
\end{align*}
Throughout this paper we assume persistence modules are finitely indexed and so morphisms can be viewed as ladder persistence modules~\cite{escolar2016ladder_modules}. For the purposes of our work, a \emph{ladder persistence module} consists of persistence modules~$V$ and~$W$ indexed by~$[l+1]$ with a family of linear maps~$\{\Phi_i\}_{i \in [l+1]}$ arranged in a commutative diagram 
\begin{center}
	\begin{tikzcd}
		V_0 \arrow[r, "{v_{0,1}}"] \arrow[d, "{\Phi_0}"] & V_1 \arrow[r, "{v_{1,2}}"] \arrow[d, "{\Phi_1}"] & V_2 \arrow[r] \arrow[d] & \cdots \arrow[r] \arrow[d] & V_{l-1} \arrow[r, "{v_{l-1,l}}"] \arrow[d, "{\Phi_{l-1}}"] & V_l \arrow[d, "{\Phi_l}"] \\
		W_0 \arrow[r, "{w_{0,1}}"']                      & W_1 \arrow[r, "{w_{1,2}}"']                      & W_2 \arrow[r]           & \cdots \arrow[r]           & W_{l-1} \arrow[r, "{w_{l-1,l}}"']            & W_l.
	\end{tikzcd}
\end{center}
Note that this definition suffices in our setting, but it can be stated more generally~\cite{escolar2016ladder_modules} in the setting of \emph{zig-zag persistence}~\cite{zigzag_og}, where the arrows of inner morphisms~$v$ and~$w$ can be reversed, as long as the direction is the same for~$v_{i, i+1}$ and~$w_{i, i+1}$ for all~$i \in [l+1]$. Since the components of a morphism~${\Phi \colon V \to W}$ between finitely indexed persistence modules commute with the inner morphisms by definition, it is easy to see that in our setting a morphisms~$\Phi$ is equivalent to a ladder persistence module, which we denote by~$(V, W, \Phi)$.

Introducing barcode bases enables us to write morphisms of persistence modules as matrices.

\begin{definition} \label{def:bar_generator}
	Let~$\{ B_i \}_{i \in [l+1]}$ be a barcode basis of persistence module~$V$ and~$J=[\alpha, \beta]$ a bar in its barcode. Then a collection~$x_J = \{b_i \in B_i \}_{i \in J}$ such that
	\begin{itemize}
		\item $A_{i+1} b_{i} = b_{i+1}$ for~$i \in [\alpha, \beta-1]$,
		\item $A_{\beta+1} b_{\beta} = 0$ and
		\item $\langle A_{\alpha} b_{\alpha-1}, b_{\alpha}\rangle_{\alpha} = 0$ for all~$b_{\alpha-1} \in B_{\alpha-1}$,
	\end{itemize}
	where the inner product~$\langle \cdot, \cdot \rangle_i$ is induced by the basis~$B_i$, is called a \emph{generator of bar~$J$}. (Note that there are~$\mu$ distinct choices for~$x_J$, where~$\mu$ is the multiplicity of bar~$J$ in~$\Barc{V}$.)
\end{definition}

\begin{proposition} \label{prop:shape_of_Phi_single_matrix}
	Choose barcode bases for persistence modules~$V$ and~$W$. Then any morphism~${\Phi\colon V \to W}$ can be written as a single matrix
	\begin{align} \label{eq:shape_of_Phi_single_matrix}
		M_\Phi = 
		\begin{bmatrix}
			X_{[0, 0]}^{[0, 0]} & X_{[0, 0]}^{[0, 1]}   & \ldots & X_{[0, 0]}^{[0, l]} & 0 & \ldots & 0 & \ldots & 0 \\
			& X_{[0, 1]}^{[0, 1]} & \ldots & X_{[0, 1]}^{[0, l]} & X_{[0,1]}^{[1,1]} & \ldots & X_{[0,1]}^{[1,l]}  & \ldots & 0 \\
			& & \ddots & \vdots & \vdots & & \vdots & & \vdots \\
			& & & X_{[0, l]}^{[0, l]} & 0  & \ldots & X_{[0,l]}^{[1,l]} & \ldots & 0 \\
			& & & & X_{[1,1]}^{[1,1]}  & \ldots & X_{[1,1]}^{[1,l]} & \ldots & 0 \\
			& & & & & \ddots & \vdots & & \vdots\\
			& & & & & & X_{[1,l]}^{[1,l]} & \ldots & 0 \\
			& & & & & & & \ddots & \vdots \\
			& & & & & & & & X_{[l,l]}^{[l,l]}
		\end{bmatrix},
	\end{align}
	where each sub-matrix~$X_{[i_2, j_2]}^{[i_1, j_1]}$ encodes how $\Phi$ maps generators of bar~$[i_1, j_1]$ in~$\Barc{V}$ to generators of bar~$[i_2, j_2]$ in~$\Barc{W}$.
\end{proposition}

The proof of \cref{prop:shape_of_Phi_single_matrix} is included in the proof of~\cite[Theorem 4.3.]{jacquard2021space} as Step $1$. Notice that a bar~$J$ with multiplicity~$\mu_J$ has~$\mu_J$ columns (or rows) associated to it, each belonging to one of the generators~$x_J$.

\begin{remark}  \label{rem:how_mphi_can_map}
	It is easy to see that there cannot be a non-zero morphism $I_{J_1} \to I_{J_2}$ between interval persistence modules unless~${J_2 \preceq J_1}$ (this must hold if the commuting-squares requirement in the definition of a morphism of persistence modules is to be satisfied). This is reflected in the general matrix shape~\eqref{eq:shape_of_Phi_single_matrix}, where the only non-zero block matrices~$X_{J_2}^{J_1}$ are associated with bars~$J_2 \preceq J_1$.
\end{remark}

Being able to write morphisms of persistence modules in a matrix also allows us to define the \emph{image}~$\Phi(x_J)$ of a bar generator~$x_J$ for~$J \in \Barc{V}$ as 
\begin{align} \label{eq:image_of_generator}
	\Phi(x_J) = \sum_{x_K} M_\Phi (x_K, x_J) \cdot x_K,
\end{align}
where the sum ranges over all generators~$x_K$ of bars~$K \in \Barc{W}$ and the scalar~$M_\Phi (x_K, x_J)$ is the entry in the row belonging to~$x_K$ and column belonging to~$x_J$ of the matrix~$M_\Phi$. In light of \cref{rem:how_mphi_can_map}, the bar generators~$x_K$ with non-zero coefficient~$M_\Phi (x_K, x_J)$ correspond to bars with~$K \preceq J$. The image~$\Phi(x_J)$ is not to be confused with the image of a basis vector in~$x_J$ with~$\Phi$, which can be expressed for each component~$\Phi_i$ as
\begin{align*}
	\Phi_i((x_J)_i) = \sum_{\substack{x_K \\ K \ni i}} M_\Phi (x_K, x_J) \cdot (x_K)_i.
\end{align*}
For example, a bar generator~$x_K$ for which~$M_\Phi (x_K, x_J)$ is non-zero has a zero coefficient in~$\Phi_i((x_J)_i)$ whenever~$i \notin K$.
\begin{definition} \label{def:support}
	The \emph{support} of the image of the bar generator~$x_J$ with morphism~$\Phi$ is the set
	\begin{align*}
		\supp \Phi (x_J) = \{ x_K \mid M_\Phi (x_K, x_J) \neq 0 \}
	\end{align*}
	of all generators~$x_K$ whose coefficient in~\eqref{eq:image_of_generator} is non-zero.
\end{definition}
\begin{lemma} \label{prop:shape_of_interleaved_barcode}
	If~$x_K \in \supp \Phi(x_{J})$ then~$K \preceq J$.
\end{lemma}
\begin{proof}
	This is a simple reiteration of the observation in~\cref{rem:how_mphi_can_map} using the new notation.
\end{proof}

Introduce the following simple and intuitive components as building blocks in the decomposition of ladder modules:~$\mathbf{R}_{[i_1, j_1] }^{[i_2, j_2]}, \mathbf{I}^+_{[i_1, j_1] }$ and~$\mathbf{I}^-_{[i_1, j_1] }$ for~$[i_1, j_1] \preceq [i_2, j_2]$.
\begin{center}
	\begin{tabular}{l}
		$\mathbf{R}_{[i_1, j_1] }^{[i_2, j_2]} \colon$
		\begin{tikzcd}
			& & \cdots \arrow[r] & 0 \arrow[r, "{0}"] \arrow[d, "{0}"] & \mathbb{F}_{i_2} \arrow[r, "{\text{id}}"]\arrow[d, "{\text{id}}"] & \cdots \arrow[r, "{\text{id}}"]\arrow[d, "{\text{id}}"] & \mathbb{F}_{j_2} \arrow[r, "{0}"] \arrow[d, "{0}"] & 0 \arrow[r] & \cdots \\
			& \cdots \arrow[r] & 0 \arrow[r, "{0}"'] & \mathbb{F}_{i_1} \arrow[r, "{\text{id}}"'] & \cdots \arrow[r, "{\text{id}}"'] & \mathbb{F}_{j_1} \arrow[r, "{0}"'] & 0 \arrow[r] & \cdots &
		\end{tikzcd} \\
		\\
		$\mathbf{I}^+_{[i_1, j_1] } \colon$
		\begin{tikzcd}
			& \cdots \arrow[r] & 0 \arrow[r, "{0}"] & \mathbb{F}_{i_1} \arrow[r, "{\text{id}}"]\arrow[d, "{0}"] & \cdots \arrow[r, "{\text{id}}"]\arrow[d] & \mathbb{F}_{j_1} \arrow[r, "{0}"] \arrow[d, "{0}"] & 0 \arrow[r] & \cdots \\
			& & \cdots \arrow[r] & 0 \arrow[r] & \cdots \arrow[r] & 0 \arrow[r] & \cdots &
		\end{tikzcd} \\
		\\
		$\mathbf{I}^-_{[i_1, j_1] }\colon$
		\begin{tikzcd}
			& & \cdots \arrow[r] & 0 \arrow[r] \arrow[d, "{0}"] & \cdots \arrow[r]\arrow[d] & 0 \arrow[r]\arrow[d, "{0}"] & \cdots & \\
			& \cdots \arrow[r] & 0 \arrow[r, "{0}"'] & \mathbb{F}_{i_1} \arrow[r, "{\text{id}}"'] & \cdots \arrow[r, "{\text{id}}"'] & \mathbb{F}_{j_1} \arrow[r, "{0}"'] & 0 \arrow[r] & \cdots &
		\end{tikzcd}
	\end{tabular}
\end{center}

\begin{theorem}[Theorem 4.3.\ of~\cite{jacquard2021space}] \label{thm:jacquard_induced_matching}
	Let~$(V, W, \Phi)$ be a ladder persistence module where neither~$V$ nor~$W$ admit a pair of strictly nested bars. Then there are integers~$r_{J_1}^{J_2}, d_{J}^{+}, d_{K}^{-} \in \mathbb{N}$ for which
	\begin{align} \label{eq:morphism_decomposition}
		(V, W, \Phi) \cong \bigoplus_{J_1 \preceq J_2}\Big( \mathbf{R}_{J_1 }^{J_2} \Big)^{r_{J_1}^{J_2}} \oplus \bigoplus_{J} \Big( \mathbf{I}^+_{J} \Big)^{d_{J}^+} \oplus \bigoplus_{K} \Big( \mathbf{I}^-_{K} \Big)^{d_{K}^-}.
	\end{align}
\end{theorem}
The right side of Equation~\eqref{eq:morphism_decomposition} is called the~\emph{ladder decomposition} of morphism~$\Phi$. The isomorphism is given by a change of barcode bases of the domain and codomain. The matrix representation of~$\Phi$ in the bases in which it decomposes as a ladder persistence module is in partial matching form.
\begin{definition} \label{def:partial_matching form}
	A matrix is in \emph{partial matching form} if there is at most one~$1$ in each row and each column, with all the other entries being~$0$.
\end{definition}
We will often refer to the partial matching form as the matching form for short. Since the ladder decomposition is unique up to an automorphism, the matching form is unique up to (compositions of) barcode basis changes acting only among generators of the same bar.

Jacquard et al.\ show with examples that if either~$\Barc{V}$ or~$\Barc{W}$ contain a pair of strictly nested bars then such a decomposition need not exist. The focus of our paper is to refine this and show that, under additional hypotheses on~$\Phi$, a decomposition will exist even when there are nested bars.

	\section{Ladder Decompositions and Interleavings} \label{sec:morphisms_and_induced_matchings}

In this section the theory of ladder decompositions of morphisms in the case of interleavings is developed further. Interleavings come in pairs of morphisms, which we (with slight abuse of notation) call~$\delta$-invertible morphisms and are interesting in their own right. \cref{thm:Xi_constant} relaxes the assumptions of the Ladder Decomposition Theorem of \cite{jacquard2021space} for~$\delta$-invertible morphisms, while \cref{cor:entries_in_matching_form} summarises the relation between ladder decompositions of the two morphisms making an interleaving pair.

\subsection{Interleavings and~\texorpdfstring{$\delta$}\ -Invertible Morphisms of Persistence Modules} \label{subsec:delta-interleavings}

Standard pseudo-metrics on the space of barcodes (or persistence diagrams) such as the bottleneck and Wasserstein distances enable us to compare barcodes of a pair of any two pointwise finite-dimensional persistence modules. However, one can also define a pseudo-metric on the space of persistence modules, called an \emph{interleaving distance}~\cite{interleavings_original_paper}. It is algebraic in nature and its definition does not require the computation of the interval decomposition. An integral role in its definition is played by~\emph{$\delta$-interleavings}.

\begin{definition}
	The \emph{shift} of a persistence module~$V$ is a persistence module defined as
	\begin{align*}
		&V(\delta)_t = V_{t+\delta} \\
		&v(\delta)_{t_1, t_2} = v_{t_1+\delta, t_2+\delta}.
	\end{align*}
\end{definition}	
\begin{definition}
	Persistence modules $V$ and $W$ are $\delta$-\emph{interleaved} if there exist morphisms~$\Phi\colon V \to W(\delta)$ and~$\Psi\colon W \to V(\delta)$ such that the diagrams
	\[
		\begin{tikzcd}
			V_t \arrow[rr, "{v_{t, t+2\delta}}"] \arrow[rd, "\Phi_t"'] & & V_{t+2\delta} \\
			& W_{t+\delta} \arrow[ru, "\Psi_{t+\delta}"'] &              
		\end{tikzcd}
		\hspace{0.5cm} and \hspace{0.5cm}
		\begin{tikzcd}
			& V_{t+\delta} \arrow[rd, "\Phi_{t+\delta}"] &               \\
			W_{t} \arrow[rr, "{w_{t, t+2\delta}}"'] \arrow[ru, "\Psi_t"] & & W_{t+2\delta}
		\end{tikzcd}
	\]
	commute. The pair $(\Phi,\Psi)$ is called a \emph{$\delta$-interleaving}.
\end{definition}
Note that two persistence modules are $0$-interleaved if and only if they are isomorphic. As a consequence, we often consider the parameter $\delta$ to measure how far from isomorphic two persistence modules can be.
\begin{definition}
	The \emph{interleaving distance}~$d_I$ is a pseudo-metric on the space of persistence modules defined with
	$$d_I(V, W) = \text{inf } \{ \delta \ | \ V\text{ and }W\text{ are }\delta \text{-interleaved} \}.$$
\end{definition}
The Isometry Theorem~\cite[Theorem~$3.5$]{bauer2013induced} states that for~$1$-parameter persistence modules the interleaving distance always equals the bottleneck distance between the barcodes. However, the interleaving distance does not require the existence of a barcode. In fact, it can be defined for any persistence module, including multiparameter ones.

\begin{example}\label{ex:running_ex_1}
	Let us introduce an interleaving which we will use as a running example. Let~$V$ and~$W$ be persistence modules
	\begin{center}
		\begin{tikzcd}
			V\colon & 0 \arrow[r]
			& \R \arrow[r, "\pinjfirst"]
			& \R^2 \arrow[r, "\text{Id}"]
			& \R^2 \arrow[r, "\text{Id}"]
			& \R^2 \arrow[r, "\vfirst"]
			& \R^3 \arrow[r, "\vsecond"]
			& \R \arrow[r, "\text{Id}"]
			& \R \arrow[r, "\text{Id}"]
			& \R \arrow[r]
			& 0, \\
			W\colon & 0 \arrow[r]
			& 0 \arrow[r]
			& \R^2 \arrow[r, "\text{Id}"]
			& \R^2 \arrow[r, "\text{Id}"]
			& \R^2 \arrow[r, "\text{Id}"]
			& \R^2 \arrow[r, "\text{Id}"]
			& \R^2 \arrow[r, " \pprojsecond"]
			& \R^1 \arrow[r]
			& 0 \arrow[r]
			& 0,
		\end{tikzcd}
	\end{center}
	with barcode bases given by the unit vectors of the vector spaces~$\R^n$.
	\begin{figure}[ht]
		\centering
		\resizebox{0.5 \linewidth}{!}{\begin{tikzpicture}
	\draw[thick,->] (0,0) -- (8,0) node[anchor=north west] {};
	\foreach \x in {0,1,2,3,4,5, 6, 7}
	\draw (\x cm,1pt) -- (\x cm,-1pt) node[anchor=north] {$\x$};
	
	\draw[black!60!green, very thick] (0, 5) -- (4, 5);
	\draw[black!60!green, very thick] (1, 4.5) -- (7, 4.5);
	\foreach \p in {(0, 5), (4, 5), (1, 4.5), (7, 4.5), (4, 4)}
	\fill[black!60!green] \p circle(.1);
	\path (8, 5) node[black!60!green]{\large $[0,4]$};
	\path (8, 4.5) node[black!60!green]{\large $[1,7]$};
	\path (8, 4) node[black!60!green]{\large $[4,4]$};
	
	\draw[white!10!blue, very thick] (1, 2) -- (5, 2);
	\draw[white!10!blue, very thick] (1, 1.5) -- (6, 1.5);
	\foreach \p in {(1, 2), (5, 2), (1, 1.5), (6, 1.5)}
	\fill[white!10!blue] \p circle(.1);
	\path (8, 2) node[white!10!blue]{\large $[1,5]$};
	\path (8, 1.5) node[white!10!blue]{\large $[1,6]$};
	
	\path (-1.4, 4.5) node[black!60!green]{\Large $\Barc{V}$};
	\path (-1.4, 1.75) node[white!10!blue]{\Large $\Barc{W}$};
\end{tikzpicture}}
		\caption{The barcodes of modules~$V$ and~$W$ in \cref{ex:running_ex_1}.}
		\label{fig:running_example}
	\end{figure}
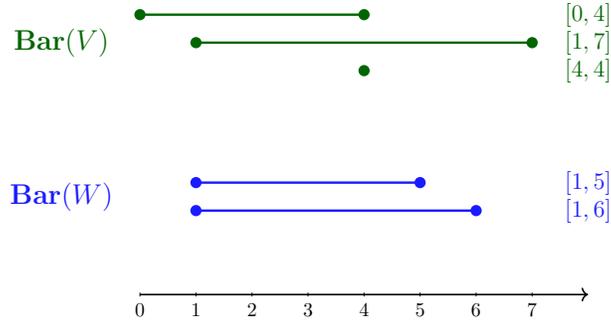
	Their barcodes are
	\begin{align*}
		\Barc{V} = \{ [0,4], [1,7], [4,4] \} \qquad \text{and} \qquad \Barc{W} = \{ [1,5], [1,6] \},
	\end{align*}
	where the bars have been written in the lexicographical order. Bar generators associated with our choice of barcode bases are
	\begin{align*}
		& x^V_{[0,4]} = \{ \vec{e_1}^{(0)}, \vec{e_1}^{(1)}, \vec{e_1}^{(2)}, \vec{e_1}^{(3)}, \vec{e_1}^{(4)} \},\\
		& x^V_{[1,7]} = \{ \vec{e_2}^{(1)}, \vec{e_2}^{(2)}, \vec{e_2}^{(3)}, \vec{e_2}^{(4)}, \vec{e_1}^{(5)}, \vec{e_1}^{(6)}, \vec{e_1}^{(7)} \}, \\
		& x^V_{[4,4]} = \{ \vec{e_3}^{(4)} \}, \\
		& \\
		& x^W_{[1,5]} =  \{ \vec{e_1}^{(1)}, \vec{e_1}^{(2)}, \vec{e_1}^{(3)}, \vec{e_1}^{(4)}, \vec{e_1}^{(5)} \}, \\
		& x^W_{[1,6]} =  \{ \vec{e_2}^{(1)}, \vec{e_2}^{(2)}, \vec{e_2}^{(3)}, \vec{e_2}^{(4)}, \vec{e_2}^{(5)}, \vec{e_1}^{(6)} \},
	\end{align*}
	where~$\vec{e}^{(i)}_j$ denotes the~$j$-th unit vector in either~$V_i$ or~$W_i$. Define morphisms~$\Phi\colon V \to W(1)$ and~${\Psi:W \to V(1)}$ with components
	\begin{align*}
		&\begin{aligned}
			\Phi_0 = \begin{bmatrix} 2 \\ 0  \end{bmatrix},
		\end{aligned}
		&&\begin{aligned}
			\Phi_{1,2,3} = \begin{bmatrix} 2 & 1 \\ 0 & 1 \end{bmatrix},
		\end{aligned}
		&&\begin{aligned}
			\Phi_4 = \begin{bmatrix} 2 & 1 & 1 \\ 0 & 1 & 0 \end{bmatrix},
		\end{aligned}
		&&\begin{aligned}
			\Phi_5 = \begin{bmatrix} 1 \end{bmatrix},
		\end{aligned}
		&&\begin{aligned}
			\Phi_{6,7} = 0,
		\end{aligned} \\
		&\begin{aligned}
			\Psi_{0,7} = 0,
		\end{aligned}
		&&\begin{aligned}
			\Psi_{1,2} = \begin{bmatrix} \half & -\half \\ 0 & 1 \end{bmatrix},
		\end{aligned}
		&&\begin{aligned}
			\Psi_{3} = \begin{bmatrix} \half & -\half \\ 0 & 1 \\ 0 & 0 \end{bmatrix},
		\end{aligned}
		&&\begin{aligned}
			\Psi_{4,5} = \begin{bmatrix} 0 & 1 \end{bmatrix},
		\end{aligned}
		&&\begin{aligned}
			\Phi_6 = \begin{bmatrix} 1 \end{bmatrix},
		\end{aligned}
	\end{align*}
	making a~$1$-interleaving pair
	\begin{center}
		\begin{tikzcd}
			V\colon & 0 \arrow[r]
			& \R \arrow[r] \arrow[rdd, white!10!blue, "\Phi_0"', near start]
			& \R^2 \arrow[r] \arrow[rdd, white!10!blue, "\Phi_1"', near start]
			& \R^2 \arrow[r] \arrow[rdd, white!10!blue, "\Phi_2"', near start]
			& \R^2 \arrow[r] \arrow[rdd, white!10!blue, "\Phi_3"', near start]
			& \R^3 \arrow[r] \arrow[rdd, white!10!blue, "\Phi_4"', near start]
			& \R \arrow[r] \arrow[rdd, white!10!blue, "\Phi_5"', near start]
			& \R \arrow[r] \arrow[rdd, white!10!blue, "\Phi_6"', near start]
			& \R \arrow[r] \arrow[rdd, white!10!blue, "\Phi_7 = 0", near end]
			& 0, \\
			&&&&&& \\
			W\colon & 0 \arrow[r] 
			& 0 \arrow[r] \arrow[ruu, black!60!green, "\Psi_0", near start]
			& \R^2 \arrow[r] \arrow[ruu, black!60!green, "\Psi_1", near start]
			& \R^2 \arrow[r] \arrow[ruu, black!60!green, "\Psi_2", near start]
			& \R^2 \arrow[r] \arrow[ruu, black!60!green, "\Psi_3", near start]
			& \R^2 \arrow[r] \arrow[ruu, black!60!green, "\Psi_4", near start]
			& \R^2 \arrow[r] \arrow[ruu, black!60!green, "\Psi_5", near start]
			& \R \arrow[r] \arrow[ruu, black!60!green, "\Psi_6", near start]
			& 0 \arrow[r] \arrow[ruu, black!60!green, "\Psi_7 = 0"', near end]
			& 0.
		\end{tikzcd}
	\end{center}
	Their respective single-matrix representations from~\cref{prop:shape_of_Phi_single_matrix} are
	\begin{align*}
		\begin{aligned}
			M_\Phi = \begin{bmatrix} 2 & 1 & 1 \\ 0 & 1 & 0  \end{bmatrix}
		\end{aligned} \qquad \text{and} \qquad
		\begin{aligned}
			M_\Psi = \begin{bmatrix} \half & -\half \\ 0 & 1 \\ 0 & 0 \end{bmatrix}.
		\end{aligned}
	\end{align*}
	We verify that these morphisms truly make an interleaving pair, by considering how their composition maps generators. Let us illustrate this on two cases in our example, the image of the generator~$x^W_{[1,6]}$ with~$\Phi(\delta = 1) \circ \Psi$ and the image of the generator~$x^V_{[4,4]}$ with~$\Psi(\delta = 1) \circ \Phi$. First, write
	\begin{align*}
		\Phi(\delta = 1) \circ \Psi(x^W_{[1,6]}) &= \Phi(\delta = 1)(x^V_{[1,7](1)} - \half x^V_{[0,4](1)}) \\
		&= x^W_{[1, 5](2)} + x^W_{[1, 6](2)} - \half (2 x^W_{[1, 5](2)}) \\
		&= x^W_{[1,6](2)}.
	\end{align*}
	where~$[i, j](\delta)$ again denotes the shifted interval~$[i-\delta, j-\delta].$
	The desired property
	\begin{align*}
		\Phi(\delta = 1) \circ \Psi(x^W_{[1,6]}) = w_2(x^W_{[1,6]}),
	\end{align*}
	where~$w_2$ is the family~$\{w_{t, t+2}\}_t$ of inner morphisms in~$W$, is easily obtained. In the second case the property
	\begin{align} \label{eq:step_in_ex1}
		\Psi(\delta = 1) \circ \Phi (x^V_{[4,4]}) = v_2 (x^V_{[4,4]}) = 0
	\end{align}
	is not as obvious. Write
	\begin{align*}
		\Psi(\delta = 1) \circ \Phi(x^V_{[4,4]}) &= \Psi(\delta = 1)(x^W_{[1,5](1)}) \\
		&= \half x^V_{[0,4](2)},
	\end{align*}
	Since the shifted interval~$[0,4](2) = [-2, 2]$ does not intersect with~$[4,4]$, the image of~$x^V_{[4,4]}$ with the composition does not contain~$x^V_{[0,4](2)}$ in its support. By discarding it, we obtain the desired Property~\eqref{eq:step_in_ex1}.
\end{example}

It is often useful to decouple the definition of an interleaving morphism by singling one of the morphisms out and implying the existence of the other without committing to a choice of it. This motivates the definition of a~\emph{$\delta$-invertible morphism}.
\begin{definition}
	A morphism~$\Phi: V \to W$ is~\emph{$\delta$-invertible} if there exists another morphism~$\Psi: W \to V(2\delta)$ such that the diagrams
	\begin{center}
		\begin{tikzcd}
			V_t \arrow[rr, "{v_{t, t+2\delta}}"] \arrow[d, "\Phi_t"'] & & V_{t+2\delta} \\
			W_{t} \arrow[rru, "\Psi_{t}"'] & &            
		\end{tikzcd}
		\hspace{0.5cm}and \hspace{0.5cm}
		\begin{tikzcd}
			& & V_{t+2\delta} \arrow[d, "\Phi_{t+2\delta}"]               \\
			W_{t} \arrow[rr, "{w_{t, t+2\delta}}"'] \arrow[rru, "\Psi_t"] & & W_{t+2\delta}
		\end{tikzcd}
	\end{center}
	commute. Morphism~$\Psi$ is called a~\emph{$\delta$-inverse} of~$\Phi$.
\end{definition}
As before, parameter~$\delta$ measures how close a morphism is to being an isomorphism, as a~$0$-invertible morphism is simply an isomorphism.

\begin{remark}[Correspondence between~$\delta$-interleavings and~$\delta$-invertible morphisms] \label{r:interleaving_vs_invertible}
	Let us make the relationship between the notions of a~$\delta$-invertible morphism and a~$\delta$-interleaving pair explicit. A morphism~$\Phi \colon V \to W$ is~$\delta$-invertible if and only if, when regarded as a morphism~$\Phi' \colon V \to W'(\delta)$ with~$W' = W(-\delta)$, it is half of a~$\delta$-interleaving. In fact, any~$\delta$-inverse of~$\Phi$ gives (after the necessary shifts) a morphism making a~$\delta$-interleaving pair with~$\Phi'$. 

\end{remark}
From here on, we will work with~$\delta$-invertible morphisms and obtain results that hold for morphisms in a~$\delta$-interleaving pair via the correspondence in \cref{r:interleaving_vs_invertible}.

\begin{example}\label{ex:running_ex_1_1}
	Take the~$1$-interleaving pair~$(\Phi, \Psi)$ from \cref{ex:running_ex_1}. The induced~$1$-invertible morphism~$\Phi^{(1)}$ maps from~$V$ to~$W'=W(1)$. Since we replace~$W$ with its shift, the barcode is now
	\begin{align*}
		\Barc{W'} = \{ [0,4], [0,5]\}
	\end{align*}
	and the bar generators in our choice of barcode basis are
	\begin{align*}
		& x^{W'}_{[0,4]} =  \{ \vec{e_1}^{(0)}, \vec{e_1}^{(1)}, \vec{e_1}^{(2)}, \vec{e_1}^{(3)}, \vec{e_1}^{(4)} \}, \\
		& x^{W'}_{[1,6]} =  \{ \vec{e_2}^{(0)}, \vec{e_2}^{(1)}, \vec{e_2}^{(2)}, \vec{e_2}^{(3)}, \vec{e_2}^{(4)}, \vec{e_2}^{(5)} \}.
	\end{align*}
	Notice, however, that the matrix representations of the components~$\Phi_i$ are equal to the matrix representations~$\Phi^{(1)}_i$, and so are the single matrix representations
	\begin{align*}
		M_\Phi = M_{\Phi^{(1)}}.
	\end{align*}
	The same holds for the~$1$-inverse~$\Psi^{(1)}$ of~$\Phi^{(1)}$.
\end{example}

\subsection{Nestedness Condition for Ladder Decomposition of a~\texorpdfstring{$\delta$}\ -Invertible Morphisms} \label{subsec:decomp_of_interleavings}

Since~$\delta$-invertible morphism are special examples of morphisms of persistence modules, \cref{thm:jacquard_induced_matching} applies to them. However, the assumption that neither~$\Barc{V}$ nor~$\Barc{W}$ admit strictly nested bars restricts the use of the theorem to a small family of morphisms. To see this clearly, observe \cref{fig:pd_nestedness}. Luckily, the special properties of~$\delta$-invertible morphism allow us to loosen the requirements of \cref{thm:jacquard_induced_matching}. In order to do that, we analyse nested bars in barcodes of persistence modules.
\begin{figure}[h!]
	\centering
	\resizebox{0.2 \linewidth}{!}{\begin{tikzpicture}[
	dot/.style = {circle, fill=black, inner sep=0pt, minimum size=5pt, node contents={}}]
	\fill[fill = blue!20!white] (1,1) -- (3,3) -- (1,3);
	\fill[fill = green!20!white] (1,3) -- (1,5) -- (0,5) -- (0,3);
	\draw[thick,->] (-0.2,0) -- (5,0);
	\draw[thick,->] (0, -0.2) -- (0, 5);
	\draw (0, 0) -- (5, 5);
	
	\path(1,3) node[dot];
\end{tikzpicture}}
	\caption{Observe a point in a persistence diagram that corresponds to a bar~$J$. Any bar~$K$ that is strictly nested with~$J$ corresponds to a point in one of the shaded regions of the diagram: if~$J \subset K$ it is in the green, and if~$K \subset J$ it is in the blue region. This illustrates that many diagrams have strictly nested points.}
	\label{fig:pd_nestedness}
\end{figure}
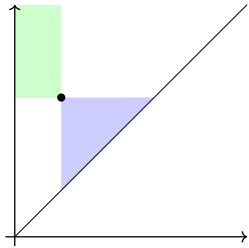

\begin{definition} \label{def:Ksi}
	For a persistence module~$M$ define a constant
	$$\Xi(M) = \min_{[a, b] \subset [c, d] \in \Barc{M}} \min \{ |a-c|,\ |b-d| \}$$
	and call it the \emph{nestedness of persistence module}~$M$. Note that the minimum loops over all pairs of strictly nested bars in~$\Barc{M}$ and is defined to equal~$\infty$ when such bars do not exist. It therefore takes values in~$(0, \infty]$.
\end{definition}
\begin{example} \label{example_nr_3}
	Let us compute nestedness for the instructive example in \cref{fig:nestedness_example_basic}. The barcode consists of bars
	\begin{align*}
		I = [0,8], \qquad J = [1, 5], \qquad K = [1, 8], \qquad L = [3, 5].
	\end{align*}
	\begin{figure}[ht]
		\centering
		\resizebox{0.5 \linewidth}{!}{\begin{tikzpicture}
	\draw[thick,->] (-1,0) -- (9,0);
	\foreach \x in {0,1,2,3,4,5, 6, 7, 8}
		\draw (\x cm,1pt) -- (\x cm,-1pt) node[anchor=north] {$\x$};

	\draw[black!60!green, very thick] (0, 4) -- (8, 4) node[midway, above]{I};
	\draw[black!60!green, very thick] (1, 2) -- (8, 2) node[midway, above]{K};
	\draw[black!60!green, very thick] (1, 3) -- (5, 3) node[midway, above]{J};
	\draw[black!60!green, very thick] (3, 1) -- (5, 1) node[midway, above]{L};
	\foreach \p in {(0,4),(8,4),(1,2),(8,2),(1,3),(5,3),(3,1),(5,1)}
		\fill[black!60!green] \p circle(.08);
\end{tikzpicture}}
		\caption{Example of a barcode with nestedness~$1$.}
		\label{fig:nestedness_example_basic}
	\end{figure}
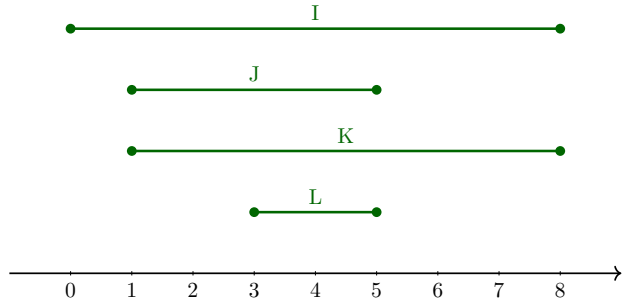
	The minimum in the definition loops over pairs~$J \subset I$,~$L \subset K$, and~$L \subset I$. The minimum distance between endpoints for each pair respectively is~$1$,~$2$ and~$3$. Perhaps not intuitively the nestedness is defined as the smallest of these values,~$1$. Two further examples of barcodes with different nestedness are shown in \cref{fig:nestedness_examples}.
	\begin{figure}[ht]
		\centering
		\subfloat[\centering Barcode with nestedness~$3$.]{{\resizebox{0.46\linewidth}{!}{
					\begin{tikzpicture}
	
	\draw[thick,->] (-1,0) -- (9,0);
	\foreach \x in {0,1,2,3,4,5, 6, 7, 8}
	\draw (\x cm,1pt) -- (\x cm,-1pt) node[anchor=north] {$\x$};
	
	\draw[black!60!green, very thick] (0, 2) -- (8, 2);
	\draw[black!60!green, very thick] (3, 1) -- (5, 1);
	\foreach \p in {(0,2),(8,2),(3,1),(5,1)}
	\fill[black!60!green] \p circle(.08);
\end{tikzpicture}} }}%
		\qquad
		\subfloat[\centering Barcode with nestedness~$\infty$.]{{\resizebox{0.46\linewidth}{!}{
				\begin{tikzpicture}
	\draw[thick,->] (-1,0) -- (9,0);
	\foreach \x in {0,1,2,3,4,5, 6, 7, 8}
	\draw (\x cm,1pt) -- (\x cm,-1pt) node[anchor=north] {$\x$};
	
	\draw[black!60!green, very thick] (0, 3) -- (7, 3);
	\draw[black!60!green, very thick] (2, 2) -- (7, 2);
	\draw[black!60!green, very thick] (4, 1) -- (8, 1);
	\foreach \p in {(0,3),(7,3),(2,2),(7,2), (4,1),(8,1)}
	\fill[black!60!green] \p circle(.08);
\end{tikzpicture}} }}%
		\caption{Further examples of barcodes with different nestedness.}%
		\label{fig:nestedness_examples}%
	\end{figure}
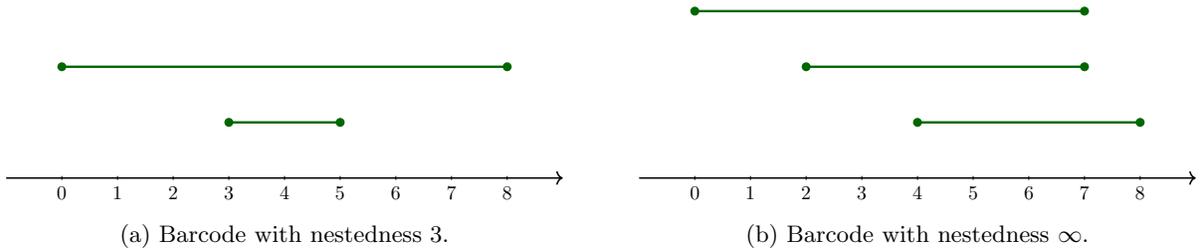
\end{example}

The main result of this paper is that, as long as the nestedness is not ``too small'', we can still obtain the ladder decomposition for a~$\delta$-invertible morphism.
\main
\begin{remark}
	Whenever there are no nested bars, i.e.\ $\min(\Xi(V), \Xi(W)) = \infty$, \cref{thm:Xi_constant} restates \cref{thm:jacquard_induced_matching}.
\end{remark}
\begin{remark}[Sufficient but not necessary condition]
	Pick your favorite example of a module~$M$ with nested bars, i.e.~$\Xi(M)< \infty$, and observe the identity morphism~$\text{Id} \colon M \to M$. It is a~$\delta$-invertible morphism for all~$\delta \in (0, \infty)$, and it does not satisfy the assumptions of \cref{thm:Xi_constant} for any~$\delta \geq \half \Xi(M)$. However, it is obvious that it admits a ladder decomposition even for those choices of~$\delta$.
\end{remark}

The proof of \cref{thm:Xi_constant} follows a similar approach as the proof of \cref{thm:jacquard_induced_matching} in~\cite{jacquard2021space}. By \cref{prop:shape_of_Phi_single_matrix}, morphism~$\Phi$ can be represented as a single block matrix
\begin{align*}
	M_\Phi = 
	\begin{bmatrix}
		X_{[0, 0]}^{[0, 0]} & X_{[0, 0]}^{[0, 1]}   & \ldots & X_{[0, 0]}^{[0, l]} & 0 & \ldots & 0 & \ldots & 0 \\
		& X_{[0, 1]}^{[0, 1]} & \ldots & X_{[0, 1]}^{[0, l]} & X_{[0,1]}^{[1,1]} & \ldots & X_{[0,1]}^{[1,l]}  & \ldots & 0 \\
		& & \ddots & \vdots & \vdots & & \vdots & & \vdots \\
		& & & X_{[0, l]}^{[0, l]} & 0  & \ldots & X_{[0,l]}^{[1,l]} & \ldots & 0 \\
		& & & & X_{[1,1]}^{[1,1]}  & \ldots & X_{[1,1]}^{[1,l]} & \ldots & 0 \\
		& & & & & \ddots & \vdots & & \vdots\\
		& & & & & & X_{[1,l]}^{[1,l]} & \ldots & 0 \\
		& & & & & & & \ddots & \vdots \\
		& & & & & & & & X_{[l,l]}^{[l,l]}
	\end{bmatrix}.
\end{align*}
This matrix will be inductively reduced to matching form. During the reduction we are allowed to use only the matrix operations whose result is the same morphism written in another barcode basis. These operations correspond to the actions of the stabilisers~$\Stab{\{v_{t, t+1}\}_t}$ and~$\Stab{\{w_{t, t+1}\}_t}$ (as in Equation~\eqref{eq:stabiliser}) on the barcode bases of the domain and codomain respectively. As a consequence of the properties of the stabilisers, namely that their elements commute with the inner morphisms, the admissible matrix operations (as deduced in~\cite{jacquard2021space}) are
	\begin{enumerate}
		\item[\namedlabel{AO1}{\textbf{AO1}}] any operation between columns corresponding to the same bar~$J$, or between rows corresponding to the same bar~$J$,
		\item[\namedlabel{AO2}{\textbf{AO2}}] modifying~$C_{J}$ using~$C_{K}$ whenever~$K \preceq J$,
		\item[\namedlabel{AO3}{\textbf{AO3}}] modifying~$R_{J}$ using~$R_{K}$ whenever~$J \preceq K$.
	\end{enumerate}

Let us introduce a few technical lemmas, which will be used in the proof.

\begin{lemma} \label{l:intersecting_support_means_intersecting_bars}
	Let~$\Phi \colon V \to W$ be a~$\delta$-invertible morphism and~$J$ and~$K$ two bars in~$\Barc{V}$. If~$\supp \Phi(x_{J})$ and~$\supp \Phi(x_{K})$ have a non-empty intersection, then~$J \cap K \neq \emptyset.$ Similarly, if generators of bars~$J$ and~$K$ in~$\Barc{W}$ both lie in~$\supp \Phi(x_{L})$ for some bar~$L \in \Barc{V}$, then~$J \cap K \neq \emptyset.$
\end{lemma}
\begin{proof}
Let us begin with the first statement. Without loss of generality, assume~$J = [c, d] \leq K =[a, b].$ By \cref{prop:shape_of_interleaved_barcode} a bar~$[i, j]$ that is in the support of both~$\Phi(x_{[c, d]})$ and~$\Phi(x_{[a, b]})$ must satisfy
\begin{align}
	& i \leq c, \nonumber\\
	& c \leq j \leq d, \label{star}\\
	& i \leq a, \text{ and} \nonumber\\
	& a \leq j \leq b. \label{starstar}
\end{align}
If~$d < a$, Inequalities~\eqref{star} and~\eqref{starstar} cannot be simultaneously satisfied.
Since such a bar~$[i, j]$ exists by assumption, the inequality~$a \leq d$ holds and the bars~$[c, d]$ and~$[a, b]$ have a non-empty intersection.

The second statement is even easier to prove. Set~$J = [c, d]$ and $K =[a, b].$ Then the bar~$L = [i, j]$ in the support of whose image~$x_{J}$ and~$x_{K}$ lie, must satisfy
\begin{align*}
	a, c \leq i \ \text{ and } \ i \leq b, d
\end{align*}
by \cref{prop:shape_of_interleaved_barcode}. It follows that~$i \in J \cap K \neq \emptyset$.
\end{proof}

\begin{lemma} \label{l:image_of_nested_bars}
	Let~$\Phi\colon V \to W$ be a~$\delta$-invertible morphism between persistence modules~$V$ and~$W$, and let~$[c, d] \subset [a, b]$ be a pair of nested bars in~$\Barc{V}$. Then any bar~$[i, j] \in \Barc{W}$ whose generator is contained in both~$\supp \Phi(x_{[a, b]})$ and~$\supp \Phi(x_{[c, d]})$ must satisfy
	\begin{align} 
		i \leq a \qquad &\text{and} \qquad c \leq j \leq d. \label{eq:first_statement} \\
		\intertext{As a consequence, the length of any such bar~$[i, j]$ must be at least~$c-a$.}
		\intertext{Similarly, any bar~$[i, j] \in \Barc{V}$ for which~$\supp \Phi(x_{[i, j]})$ contains generators~$x_{[a, b]}$ and~$x_{[c, d]}$ of nested bars~${[c, d] \subset [a, b] \in \Barc{W}}$ must satisfy}
		c \leq i \leq d \qquad &\text{and} \qquad b \leq j. \label{eq:second_statement}
	\end{align}
	As a consequence, the length of any such bar~$[i, j]$ must be at least~$b-d$.
\end{lemma}

\begin{proof}
	Both statements are a simple consequence of \cref{prop:shape_of_interleaved_barcode}. For the bars in~\eqref{eq:first_statement} it states that
	\begin{align*}
		i \leq a \leq j \leq b \qquad &\text{and} \qquad i \leq c \leq j \leq d \\
		\intertext{(observe \cref{subfig:a_image_preimage_of_nested_bars}). Since the bars are nested, these requirements can be summarised as}
		i \leq a \qquad &\text{and} \qquad c \leq j \leq d.
	\end{align*}
	It follows readily that~$j-i \geq c-a$.
	\begin{figure}[t]
		\centering
		\begin{subfigure}[t]{.47\textwidth}
			\centering
			\resizebox{\linewidth}{!}{\begin{tikzpicture}[
	dot/.style = {circle, fill=black, inner sep=0pt, minimum size=5pt, node contents={}},
	dotorange/.style = {circle, fill=white!10!blue, inner sep=0pt, minimum size=5pt, node contents={}},
	bargreen/.style = {very thick, black!60!green},
	barorange/.style = {very thick, white!10!blue},
	pomozna/.style = {very thin, gray!50!, dashed},
	]
	
	\fill[green!20!white] (-2,-5.5) rectangle (0,-3.5);
	\fill[blue!15!white] (2,-5.5) rectangle (6,-3.5);
	
	\path (-1,-4.5) node[black!60!green]{$birth$};
	\path (4,-4.5) node[white!10!blue]{$death$};
	
	\coordinate (a) at (0,0);
	\coordinate (b) at (7,0);
	\coordinate (c) at (2,-1);
	\coordinate (d) at (6,-1);
	\coordinate (cc) at (2, -5.5);
	\coordinate (dd) at (6,-5.5);
	\coordinate (aa) at (0,-5.5);
	\coordinate (bb) at (7,-5.5);
	
	\draw[pomozna] (c) -- (cc);
	\draw[pomozna] (d) -- (dd);
	\draw[pomozna] (a) -- (aa);
	\draw[pomozna] (b) -- (bb);
	
	\draw (a) -- (b);
	\draw (c) -- (d);
	\path (c) node[dot,label=above:$c$];
	\path (d) node[dot,label=above:$d$];
	\path (b) node[dot,label=above:$b$];
	\path (a) node[dot,label=above:$a$];
	
	\path (0,-4) node[black]{$\Big]$};
	\path (0,-4) node[black]{$\Big[$};
	\path (0,-3.4) node[black]{$a$};
	\draw[bargreen] (-2, -4) -- (0,-4);
	\path (7,-4) node[black]{$\Big]$};
	\path (7,-3.4) node[black]{$b$};
	\draw[barorange] (0,-4) -- (7, -4);

	\path (2,-5) node[black]{$\Big]$};
	\path (2,-5) node[black]{$\Big[$};
	\path (2,-4.4) node[black]{$c$};
	\draw[bargreen] (-2, -5) -- (2,-5);
	\path (6,-5) node[black]{$\Big]$};
	\path (6,-4.4) node[black]{$d$};
	\draw[barorange] (2,-5) -- (6, -5);
	
	\draw[stealth-stealth, black!30!red] (0, -2) -- (2, -2); 
	\path (1,-1.7) node[black!30!red]{$c-a$};
\end{tikzpicture}}
			\caption{The restrictions for a bar in the support of both~$\Phi(x_{[a, b]})$ and~$\Phi(x_{[c, d]})$, where~$\Phi$ is a~$\delta$-invertible morphism and~$[c, d] \subset [a, b]$ are strictly nested bars. The restrictions for birth point are marked in green, while the restrictions for death point are marked in blue, with the second line from the bottom showing the restrictions induced by bar~$[a, b]$ and the bottom one those induced by~$[c, d]$. The interval in which both are satisfied is marked with a square in the respective colour.}
			\label{subfig:a_image_preimage_of_nested_bars}
		\end{subfigure}
		\hfill
		\begin{subfigure}[t]{.47\textwidth}
			\centering
			\resizebox{\linewidth}{!}{\begin{tikzpicture}[
	dot/.style = {circle, fill=black, inner sep=0pt, minimum size=5pt, node contents={}},
	dotorange/.style = {circle, fill=white!10!blue, inner sep=0pt, minimum size=5pt, node contents={}},
	bargreen/.style = {very thick, black!60!green},
	barorange/.style = {very thick, white!10!blue},
	pomozna/.style = {very thin, gray!50!, dashed},
	]
	
	\fill[green!20!white] (2,2.5) rectangle (6,4.5);
	\fill[blue!15!white] (7,2.5) rectangle (9,4.5);
	
	\path (4,3.5) node[black!60!green]{$birth$};
	\path (8.05,3.5) node[white!10!blue]{$death$};
	
	\coordinate (a) at (0,0);
	\coordinate (b) at (7,0);
	\coordinate (c) at (2,-1);
	\coordinate (d) at (6,-1);
	\coordinate (cc) at (2, 4.5);
	\coordinate (dd) at (6,4.5);
	\coordinate (aa) at (0,4.5);
	\coordinate (bb) at (7,4.5);
	
	\draw[pomozna] (c) -- (cc);
	\draw[pomozna] (d) -- (dd);
	\draw[pomozna] (a) -- (aa);
	\draw[pomozna] (b) -- (bb);
	
	\draw (a) -- (b);
	\draw (c) -- (d);
	\path (c) node[dot,label=above:$c$];
	\path (d) node[dot,label=above:$d$];
	\path (b) node[dot,label=above:$b$];
	\path (a) node[dot,label=above:$a$];
	
	\path (0,4) node[black]{$\Big[$};
	\path (0,4.6) node[black]{$a$};
	\draw[bargreen] (0, 4) -- (7,4);
	\path (7,4) node[black]{$\Big]$};
	\path (7,4) node[black]{$\Big[$};
	\path (7,4.6) node[black]{$b$};
	\draw[barorange] (7,4) -- (9, 4);

	\path (2,3) node[black]{$\Big[$};
	\path (2,3.6) node[black]{$c$};
	\draw[bargreen] (2,3) -- (6, 3);
	\path (6,3) node[black]{$\Big]$};
	\path (6,3.6) node[black]{$d$};
	\draw[barorange] (6,3) -- (9, 3);
	
	\draw[stealth-stealth, black!30!red] (6, 1) -- (7, 1); 
	\path (6.5, 1.6) node[black!30!red]{$b-d$};
\end{tikzpicture}}
			\caption{The restrictions for a bar with both~$x_{[a, b]}$ and~$x_{[c, d]}$ in the support of its image with a~$\delta$-invertible morphism, where~$[c, d] \subset [a, b]$ are strictly nested bars. The restriction for birth point are marked in green, while the restrictions for death point are marked in blue, with the top line showing the restrictions induced by bar~$[a, b]$ and the second line from the top showing those induced by~$[c, d]$. The interval in which both are satisfied is marked with a square in the respective colour.}
			\label{subfig:b_image_preimage_of_nested_bars}
		\end{subfigure}
		\label{fig:image_preimage_of_nested_bars}
		\caption{The implications of \cref{prop:shape_of_interleaved_barcode} for a pair of nested bars~$[c, d] \subset [a, b]$ in~$\Barc{V}$ (case (a)) or~$\Barc{W}$ (case (b)).}
	\end{figure}
	Statement~(\ref{eq:second_statement}) can be proved in a similar way (observe \cref{subfig:b_image_preimage_of_nested_bars}).
\end{proof}

\begin{lemma} \label{l:existence_of_matched_bar}
	Let~$\Phi\colon V \to W$ be a~$\delta$-invertible morphism and~$\Psi:W \to V(2\delta)$ its~$\delta$-inverse. For any generator~$x_{[a, b]}$ of a bar~$[a, b] \in \Barc{V}$ of length at least~$2\delta$ there exists at least one generator~$x_{[i, j]}$ of bar~$[i, j] \in \Barc{W}$ such that
	\begin{align}
		x_{[i, j]} \in \supp \Phi(x_{[a, b]}) \label{eq:line1_supp},\\
		x_{[a, b]}(2\delta) \in \supp \Psi(x_{[i, j]}). \label{eq:line2_supp}
	\end{align}
	Similarly, for any generator~$x_{[i, j]}$ of a bar~$[i, j] \in \Barc{W}$ of length at least~$2\delta$ there exists at least one generator~$x_{[a, b]}$ of bar~$[a, b] \in \Barc{V}$ such that~\eqref{eq:line1_supp} and~\eqref{eq:line2_supp} hold.
	
	Bars~$[a, b] \in \Barc{V}$ and~$[i, j] \in \Barc{W}$ related in this way must satisfy
	\begin{align*}
		a-2\delta \leq i \leq a \leq i+2\delta \qquad \text{and} \qquad b-2\delta \leq j \leq b \leq j+2\delta.
	\end{align*}
\end{lemma}

\begin{proof}
	The containment statements are a simple consequence of the fact that the compositions~$\Phi \circ \Psi$ and~$\Psi \circ \Phi$ map generators belonging to bars of length at least~$2\delta$ to themselves. The rest is a consequence of applying \cref{prop:shape_of_interleaved_barcode} to~\eqref{eq:line1_supp} and~\eqref{eq:line2_supp} and combining the obtained inequalities.
\end{proof}

\begin{remark} \label{r:matches_of_long_bars}
	Given a~$\delta$-interleaving pair~$(\Phi, \Psi)$ between~$V$ and~$W$, \cref{l:existence_of_matched_bar} implies that for any generator~$x_{[a, b]}$ of a bar~$[a, b] \in \Barc{V}$ with~$b-a \geq 2\delta$ there exists a generator~$x_{[i, j]}$ of bar~$[i, j] \in \Barc{W}$ such that
	\begin{align}
		x_{[i, j]}(\delta) \in \supp \Phi(x_{[a, b]}),\\
		x_{[a, b]}(\delta) \in \supp \Psi(x_{[i, j]}).
	\end{align}
	Further, bar~$[i, j]$ must satisfy
	\begin{align*}
		|i - a| \leq \delta \qquad \text{and} \qquad |j-b| \leq \delta.
	\end{align*}
	Similar holds for any generator~$x_{[i, j]}$ of a bar with length at least~$2\delta$.
\end{remark}

\begin{lemma}\label{l:induction_step}
	Let~$M_\Phi$ be the matrix of a~$\delta$-invertible morphism~$\Phi\colon V \to W$ for~${\delta < \half \min(\Xi(V), \Xi(W))}$ written in barcode bases~$\mathcal{B}_V$ and~$\mathcal{B}_W$. Let~$X_{[i, j]}^{[c, d]}$ be a sub-matrix of~$M_\Phi$ such that all the sub-matrices to the left and below it have already been reduced to matching form with operations~\ref{AO1},~\ref{AO2} and~\ref{AO3}. Then~$X_{[i, j]}^{[c, d]}$ can also be reduced using these operations.
\end{lemma}

\begin{proof}	
	Let~$A$ denote the sub-matrix containing all the sub-matrices appearing to the left and downward of~$X_{[i, j]}^{[c, d]}$ (observe \cref{fig:induction_step}). Since all other sub-matrices in~$A$ have already been reduced there is at most one~$1$ in each row and column of~$A$ outside of~$X_{[i, j]}^{[c, d]}$. A non-zero entry of~$X_{[i, j]}^{[c, d]}$ in row~$R$ and column~$C$ falls in (at least) one of the following categories:
	\begin{enumerate}
		\item The entries of row~$R$ and column~$C$ in~$A$, that are not in~$X_{[i, j]}^{[c, d]}$, are zero.
		\item There is a~$1$ in the row~$R$ to the left of~$X_{[i, j]}^{[c, d]}$.
		\item There is a~$1$ in the column~$C$ bellow~$X_{[i, j]}^{[c, d]}$.
	\end{enumerate}
	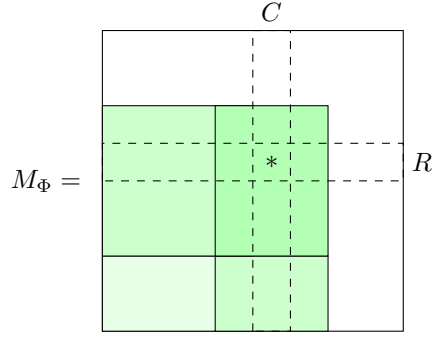
\begin{figure}[ht]
		\centering
%
%

\begin{tikzpicture}
	\draw[] (1,1) rectangle (5,5);
	\draw[fill = green!10!white] (1,1) rectangle (3,2);
	\draw[fill = green!20!white] (2.5, 1) rectangle (4,2);
	\draw[fill = green!20!white] (1, 2) rectangle (3,4);
	\draw[fill = green!30!white] (2.5, 2) rectangle (4,4);
	\path (3.25, 3.25) node[black]{$*$};
	\draw[dashed] (1,3) rectangle (5,3.5);
	\draw[dashed] (3,1) rectangle (3.5,5);
	\path (0.25, 3) node[black]{$M_\Phi =$};
	\path (5.25, 3.25) node[black]{$R$};
	\path (3.25, 5.25) node[black]{$C$};
\end{tikzpicture}
		\caption{Areas in green denote the sub-matrix~$A$ containing all sub-matrices to the left and below sub-matrix~$X_{[i, j]}^{[c, d]}$, marked with the darkest shade. The non-zero entry~$M_\Phi (R, C)$ is denoted with~$*$.}
		\label{fig:induction_step}
	\end{figure}
	When reducing~$X_{[i, j]}^{[c, d]}$, start with entries of type~$2$ and~$3$. We will justify that we can set them to zero using matrix operations of type~\ref{AO2} and~\ref{AO3}. After this is done, the non-zero entries will all be of type~$1$. The remaining steps in the reduction can be performed using operations of type~\ref{AO1}.
	
	To show we can set an entry of type~$2$ to zero, assume~$1$ in row~$R$ lies in a column belonging to a bar~$[a, b]$. Let us prove, that~$[a, b] \preceq [c, d]$ and we can use~\ref{AO2} to reduce~$M_\Phi(R, C)$. We already know that~$[a, b] \leq [c, d]$, and by \cref{l:intersecting_support_means_intersecting_bars} the bars~$[a, b]$ and~$[c, d]$ have a non-empty intersection. It remains to be proven that~$[c, d] \not\subset [a, b]$. For that purpose, assume~$[c, d] \subset [a, b]$ and observe \cref{subfig:a_lemma317}. By \cref{l:image_of_nested_bars} the bar~$[i, j]$, which is in the support of both~$\Phi(x_{[c, d]})$ and~$\Phi(x_{[a, b]})$, must satisfy
	\allowdisplaybreaks
	\begin{align}
		i \leq a \qquad &\text{and} \qquad c \leq j \leq d, \label{eq:ij_restrictions} \\
		\intertext{and be of length at least~$c-a \geq \Xi(V) > 2\delta.$ By \cref{l:existence_of_matched_bar} there exists a bar~$[k, l] \in \Barc{V}$ for which}
		x_{[i, j]} \in \supp \Phi(x_{[k, l]}) \qquad &\text{and} \qquad x_{[k, l]}(2\delta) \in \supp \Psi(x_{[i, j]}), \label{eq:mutual_inclusion_in_image} \\
		i \leq k \leq i +2\delta \qquad &\text{and} \qquad j \leq l \leq j+2\delta, \label{eq:kl_restriction}
		\intertext{where~$\Psi:W \to V(2\delta)$ is an arbitrary choice of a~$\delta$-inverse of~$\Phi$. The combined inequalities~\eqref{eq:ij_restrictions} and~\eqref{eq:kl_restriction} give us restrictions on~$[k, l]$:}	
		k \leq a+2\delta \qquad &\text{and} \qquad c \leq l \leq d+2\delta. \nonumber
	\end{align}
	Because~$\Xi(V) > 2\delta$ we further have~$l \leq d+2\delta < b$. Now notice that~$k$ cannot be larger than~$a$: if it is, then~$[k, l] \subset [a, b]$ and~$k-a \leq 2\delta$, which violates the assumption that~$\delta < \half \Xi(V)$. Similarly,~$l$ cannot be larger than~$d$, since~$d < l \leq d+2\delta$ implies~$[c, d] \subset [k, l]$, which violates the same assumption. As a consequence of these observations, the bar~$[k, l]$ cannot be equal to~$[a, b]$ or~$[c, d]$.
	
	To summarise, there is a bar~$[k, l] \neq [a, b], [c, d]$ satisfying~\eqref{eq:mutual_inclusion_in_image}, which appears before~$[a, b]$ and~$[c, d]$ in the order~$\leq.$ Because~$x_{[i, j]} \in \supp \Phi(x_{[k, l]})$, the entry in row~$R$ and column belonging to~$[k, l]$ must be non-zero. This means there are two non-zero entries in row~$R$ of sub-matrix~$A$ to the left of~$X_{[i, j]}^{[c, d]}$, which cannot be true, since all sub-matrices in~$A$ except for~$X_{[i, j]}^{[c, d]}$ are reduced. Since we obtained a contradiction, bars~$[c, d]$ and~$[a, b]$ are not strictly nested. We can use operations of type~\ref{AO2} to reduce the entry~$M_\Phi(R, C)$.
	
		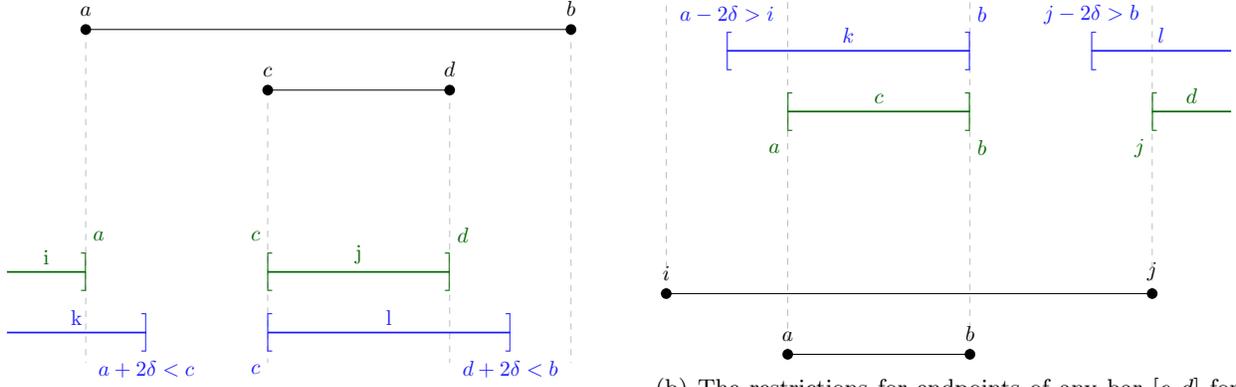
\begin{figure}[ht]
		\centering
		\begin{subfigure}[ht]{.47\textwidth}
			\centering
			\resizebox{\linewidth}{!}{\begin{tikzpicture}[
	dot/.style = {circle, fill=black, inner sep=0pt, minimum size=5pt, node contents={}},
	dotorange/.style = {circle, fill=white!10!blue, inner sep=0pt, minimum size=5pt, node contents={}},
	bargreen/.style = {thick, black!60!green},
	barorange/.style = {thick, white!10!blue},
	pomozna/.style = {very thin, gray!50!, dashed},
	zelena/.style = {black!60!green},
	rdeca/.style = {white!10!blue}
	]
	
	\coordinate (a) at (-1,0);
	\coordinate (b) at (7,0);
	\coordinate (c) at (2,-1);
	\coordinate (d) at (5,-1);
	\coordinate (cc) at (2, -5.5);
	\coordinate (dd) at (5,-5.5);
	\coordinate (aa) at (-1,-5.5);
	\coordinate (bb) at (7,-5.5);
	
	\draw[pomozna] (c) -- (cc);
	\draw[pomozna] (d) -- (dd);
	\draw[pomozna] (a) -- (aa);
	\draw[pomozna] (b) -- (bb);
	
	\draw (a) -- (b);
	\draw (c) -- (d);
	\path (c) node[dot,label=above:$c$];
	\path (d) node[dot,label=above:$d$];
	\path (b) node[dot,label=above:$b$];
	\path (a) node[dot,label=above:$a$];
	
	\path (-1,-4) node[zelena]{$\Big]$};
	\path (-1,-3.4) node[zelena, anchor=west]{$a$};
	\path (2,-4) node[zelena]{$\Big[$};
	\path (2,-3.4) node[zelena, anchor=east]{$c$};
	\path (5,-4) node[zelena]{$\Big]$};
	\path (5,-3.4) node[zelena, anchor=west]{$d$};
	\draw[bargreen] (-2.3, -4) -- (-1,-4) node[midway, above]{i};
	\draw[bargreen] (2,-4) -- (5, -4) node[midway, above]{j};
	
	\path (0,-5) node[rdeca]{$\Big]$};
	\path (0,-5.6) node[rdeca]{$a+2\delta<c$};
	\path (2,-5) node[rdeca]{$\Big[$};
	\path (2,-5.6) node[rdeca, anchor=east]{$c$};
	\path (6,-5) node[rdeca]{$\Big]$};
	\path (6,-5.6) node[rdeca]{$d+2\delta<b$};
	\draw[barorange] (-2.3, -5) -- (0,-5) node[midway, above]{k};
	\draw[barorange] (2,-5) -- (6, -5) node[midway, above]{l};
\end{tikzpicture}}
			\caption{The restrictions for endpoints of any bar~$[i, j]$ whose generator is in the support of~$\Phi(x_{[a, b]})$ and~$\Phi(x_{[c, d]})$ for a~$\delta$-invertible morphism~$\Phi$ are shown in green. In blue are the restrictions for the endpoints of a bar~$[k, l]$ satisfying~\eqref{eq:mutual_inclusion_in_image} and~\eqref{eq:kl_restriction}, which exists by \cref{l:existence_of_matched_bar}.}
			\label{subfig:a_lemma317}
		\end{subfigure}
		\hfill
		\begin{subfigure}[ht]{.47\textwidth}
			\centering
			\resizebox{\linewidth}{!}{\begin{tikzpicture}[
	dot/.style = {circle, fill=black, inner sep=0pt, minimum size=5pt, node contents={}},
	dotorange/.style = {circle, fill=white!10!blue, inner sep=0pt, minimum size=5pt, node contents={}},
	bargreen/.style = {thick, black!60!green},
	barorange/.style = {thick, white!10!blue},
	pomozna/.style = {very thin, gray!50!, dashed},
	zelena/.style = {black!60!green},
	rdeca/.style = {white!10!blue}
	]

	\coordinate (a) at (0,0);
	\coordinate (b) at (8,0);
	\coordinate (c) at (2,-1);
	\coordinate (d) at (5,-1);
	\coordinate (cc) at (2, 5);
	\coordinate (dd) at (5,5);
	\coordinate (aa) at (0,5);
	\coordinate (bb) at (8,5);
	
	\draw[pomozna] (c) -- (cc);
	\draw[pomozna] (d) -- (dd);
	\draw[pomozna] (a) -- (aa);
	\draw[pomozna] (b) -- (bb);
	
	\draw (a) -- (b);
	\draw (c) -- (d);
	\path (c) node[dot,label=above:$a$];
	\path (d) node[dot,label=above:$b$];
	\path (b) node[dot,label=above:$j$];
	\path (a) node[dot,label=above:$i$];
	
	\path (1,4) node[rdeca]{$\Big[$};
	\path (1,4.6) node[rdeca]{$a-2\delta > i$};
	\draw[barorange] (1, 4) -- (5,4) node[midway, above]{$k$};
	\path (5,4) node[rdeca]{$\Big]$};
	\path (5,4.6) node[rdeca, anchor=west]{$b$};
	\path (7,4) node[rdeca]{$\Big[$};
	\path (7,4.6) node[rdeca]{$j-2\delta >b$};
	\draw[barorange] (7,4) -- (9.3, 4)node[midway, above]{$l$};

	\path (2,3) node[zelena]{$\Big[$};
	\path (2,2.4) node[zelena, anchor=east]{$a$};
	\path (5,3) node[zelena]{$\Big]$};
	\path (5,2.4) node[zelena, anchor=west]{$b$};
	\draw[bargreen] (2,3) -- (5, 3) node[midway, above]{$c$};
	\path (8,3) node[zelena]{$\Big[$};
	\path (8,2.4) node[zelena, anchor=east]{$j$};
	\draw[bargreen] (8,3) -- (9.3, 3) node[midway, above]{$d$};
\end{tikzpicture}}
			\caption{The restrictions for endpoints of any bar~$[c, d]$ for which the support~$\Phi(x_{[c, d]})$ contains generators~$x_{[a, b]}$ and~$x_{[i, j]}$ for a~$\delta$-invertible morphism~$\Phi$ are shown in green. In blue are the restrictions for the endpoints of a bar~$[k, l]$ satisfying~\eqref{eq:mutual_inclusion_in_image_2} and~\eqref{eq:kl_restriction_2}, which exists by \cref{l:existence_of_matched_bar}.}
			\label{subfig:b_lemma317}
		\end{subfigure}
		\label{fig:lemma317}
		\caption{Restrictions for endpoints of bars used in the proof of \cref{l:induction_step}.}
	\end{figure}
	
	The fact that the entries of the third type can be set to zero using operations of type~\ref{AO3} can be proven in a similar way. Assume~$1$ in column~$C$ lies in a row belonging to a bar~${[a, b] \in \Barc{W}}$. We know that~${[i, j] \leq [a, b]}$, and by \cref{l:intersecting_support_means_intersecting_bars} the bars~$[a, b]$ and~$[i, j]$ have a non-empty intersection. As before, assume~$[a, b] \subset [i, j]$ and observe \cref{subfig:b_lemma317}. By \cref{l:image_of_nested_bars} the bar~$[c, d] \in \Barc{V}$, for which~$\supp \Phi(x_{[c, d]})$ contains both~$x_{[i, j]}$ and~$x_{[a, b]}$, must satisfy 
	\begin{align}
		a \leq c \leq b \qquad &\text{and} \qquad j \leq d, \label{eq:ij_restrictions_2} \\
		\intertext{and be of length at least~$j-b \geq \Xi(V) > 2\delta.$ By \cref{l:existence_of_matched_bar} there exists a bar~$[k, l] \in \Barc{W}$ for which}
		x_{[k, l]} \in \supp \Phi(x_{[c, d]}) \qquad &\text{and} \qquad x_{[c, d]}(2\delta) \in \supp \Psi(x_{[k, l]}), \label{eq:mutual_inclusion_in_image_2} \\
		c-2\delta \leq k \leq c \qquad &\text{and} \qquad d-2\delta \leq l \leq d. \label{eq:kl_restriction_2}
	\end{align}
	Notice that~$l$ cannot be smaller than~$j$: if it is, the combined restrictions~\eqref{eq:ij_restrictions_2} and~\eqref{eq:kl_restriction_2} give us~$j-2\delta \leq l < j$ and~$a-2\delta \leq k \leq b$. As a consequence~$[k, l] \subset [i, j]$ for~$|j-l| < 2\delta$, which violates the assumption that~$\delta < \half \Xi(V)$. Similarly,~$k$ cannot be smaller than~$a$, since~$a-2\delta \leq k < a$ implies~$[a, b] \subset [k, l]$, which violates the same assumption.
	
	To summarise, there is a bar~$[k, l] \neq [a, b], [i, j]$ satisfying~\eqref{eq:mutual_inclusion_in_image_2}, which appears after~$[a, b]$ and~$[i, j]$ in the order~$\leq.$ Because~$x_{[k, l]} \in \supp \Phi(x_{[c, d]})$, the entry in column~$C$ and row belonging to~$[k, l]$ must be non-zero. This means there are two non-zero entries in column~$C$ of sub-matrix~$A$ below~$X_{[i, j]}^{[c, d]}$, which cannot be, since all sub-matrices in~$A$ except for~$X_{[i, j]}^{[c, d]}$ are reduced. Since we obtained a contradiction, bars~$[i, j]$ and~$[a, b]$ are not strictly nested. We can therefore use the row belonging to bar~$[a, b]$ in the reduction of~$M_\Phi(R, C)$.
\end{proof}

	We are now ready to conclude the proof of \cref{thm:Xi_constant}.

\begin{proof}[Proof of \cref{thm:Xi_constant}]
	Let us reduce the matrix~$M_\Phi$ to matching form by admissible operations~\ref{AO1},~\ref{AO2} and~\ref{AO3}. The reduction is done on sub-matrices inductively, processing the columns from left to right, starting at the lowest non-zero sub-matrix in each column and continuing upwards. Choosing this order, the assumptions of \cref{l:induction_step} are satisfied at each step, including for the first sub-matrix in the order. As a consequence the fact that the whole matrix can be reduced  using admissible operations~\ref{AO1},~\ref{AO2} and~\ref{AO3} follows readily.
\end{proof}
As is the case for the general morphism, the matching form of~$M_\Phi$ is unique up to barcode basis changes acting among different bar generators of the same bar.

\begin{example} \label{ex:additional_assumption}
	The inequality in the assumption~$\ass$ of \cref{thm:Xi_constant} is strict. Here, we provide an example of a~$\delta$-invertible morphism~$\Phi \colon V \to W$ for which~${2\delta = \min(\Xi(V), \Xi(W))}$ that does not admit a ladder decomposition. Let~$V$ and~$W$ be persistence modules with barcodes~$\Barc{V} = \{ [0,7], [2,5] \}$ and~$\Barc{W} = \{ [0,4], [0,5] \}$, and barcode bases given by generators
	\begin{align*}
		x_{[0,7]}^V &= \{ \vec{e_1}^{(0)}, \vec{e_1}^{(1)}, \vec{e_1}^{(2)}, \vec{e_1}^{(3)}, \vec{e_1}^{(4)}, \vec{e_1}^{(5)}, \vec{e_1}^{(6)}, \vec{e_1}^{(7)} \}, \\
		x_{[2,5]}^V &= \{ \vec{e_2}^{(2)}, \vec{e_2}^{(3)}, \vec{e_2}^{(4)}, \vec{e_2}^{(5)} \}, \\
		x_{[0,4]}^W &= \{ \vec{e_1}^{(0)}, \vec{e_1}^{(1)}, \vec{e_1}^{(2)}, \vec{e_1}^{(3)}, \vec{e_1}^{(4)} \}, \\
		x_{[0,5]}^W &= \{ \vec{e_2}^{(0)}, \vec{e_2}^{(1)}, \vec{e_2}^{(2)}, \vec{e_2}^{(3)}, \vec{e_2}^{(4)}, \vec{e_1}^{(5)} \}.
	\end{align*}
	Define the~$1$-invertible morphism~$\Phi$ and its~$1$-inverse~$\Psi \colon V \to W(2)$ by giving their matrix representations
	\begin{align*}
		\begin{aligned}
			M_{\Phi} = \begin{bmatrix} \half & -\half \\[0.5em]
				\half & 1  \end{bmatrix}
		\end{aligned} \qquad \text{and} \qquad
		\begin{aligned}
			M_{\Psi} = \begin{bmatrix} \half & -\half \\[0.5em]
				\half & 1  \end{bmatrix}.
		\end{aligned}
	\end{align*}
	Notice that~$\Xi(V) = 2$ and~$\Xi(W) = \infty$, and so~$\delta=1$ satisfies~$2\delta = \min(\Xi(V), \Xi(W))$. Further, since the two bars constituting the barcode of~$V$ are nested, the only allowed operation on the columns of~$M_\Phi$ is scaling. The allowed operations on rows are scaling and adding a multiple of the bottom row to the top one, but these are not enough to reduce~$M_\Phi$ to matching form.
\end{example}

\begin{example} \label{ex:running_ex_2}
	Consider again the~$1$-invertible morphism~$\Phi^{(1)}\colon V \to W'$ from \cref{ex:running_ex_1_1}. Remember, the barcodes of the modules are 
	\begin{align*}
		\Barc{V} = \{ [0,4], [1,7], [4,4] \} \qquad \text{and} \qquad \Barc{W'} = \{ [0,4], [0,5] \},
	\end{align*}
	where the bars have been written in the lexicographical order, and the single-matrix representation of~$\Phi^{(1)}$ in the barcode bases chosen in \cref{ex:running_ex_1} and \cref{ex:running_ex_1_1} is
	\begin{align*}
		\begin{aligned}
			M_{\Phi^{(1)}} = \begin{bmatrix} 2 & 1 & 1 \\ 0 & 1 & 0  \end{bmatrix}.
		\end{aligned}
	\end{align*}
	Note that the nestedness of modules~$V$ and~$W'$ are~$3$ and~$\infty$ respectively, and since~$\delta < \half \min (3, \infty)$ holds for the shifting parameter~$\delta = 1$, morphism~$\Phi^{(1)}$ satisfies the assumptions of \cref{thm:Xi_constant}.
	Let us reduce the matrix~$M_{\Phi^{(1)}}$ to matching form while keeping track of bases changes. We begin with the entry~$M_{\Phi^{(1)}}(1, 1) = 2$, which is reduced by an admissible operation of type~\ref{AO1},
	\begin{align*}
		x^V_{[0,4]} \mapsto \half x^V_{[0,4]}.
	\end{align*}
	Moving on to the second column, leave the entry~$M_{\Phi^{(1)}}(2,2) = 1$ unchanged and eliminate~$M_{\Phi^{(1)}}(1, 2) = 1$ by subtracting the second row from the first, which is the admissible operation of type~\ref{AO3} since~$[0,4] \preceq [0,5]$. This corresponds to the basis change 
	\begin{align*}
		x^{W'}_{[0,5]} \mapsto x^{W'}_{[0,5]} + x^{W'}_{[0,4]}.
	\end{align*}
	The updated matrix~$M_{\Phi^{(1)}}$ is now
	\begin{align*}
		\begin{aligned}
			M_{\Phi^{(1)}} = \begin{bmatrix} 1 & 0 & 1 \\ 0 & 1 & 0  \end{bmatrix}.
		\end{aligned}
	\end{align*}
	To eliminate~$M_{\Phi^{(1)}}(1,3)$ and finish the reduction process, perform the basis change
	\begin{align*}
		x^V_{[4,4]} \mapsto x^V_{[4,4]}-x^V_{[0,4]},
	\end{align*}
	which corresponds to the admissible operation of type~\ref{AO2} subtracting first column from the last.
 	The ladder decomposition of~$(V, W', \Phi^{(1)})$ is therefore
 	\begin{align*}
 		\mathbf{R}^{[0,4]}_{[0,4]} \oplus \mathbf{R}^{[1,7]}_{[0,5]} \oplus \mathbf{I}^+_{[4,4]}
 	\end{align*}
 	and is obtained in barcode bases in which the bar generators are
 	\begin{align*}
 		& x^V_{[0,4]} = \{ \half \vec{e_1}^{(0)}, \half \vec{e_1}^{(1)}, \half \vec{e_1}^{(2)}, \half \vec{e_1}^{(3)}, \half \vec{e_1}^{(4)} \},\\
 		& x^V_{[1,7]} = \{ \vec{e_2}^{(1)}, \vec{e_2}^{(2)}, \vec{e_2}^{(3)}, \vec{e_2}^{(4)}, \vec{e_1}^{(5)}, \vec{e_1}^{(6)}, \vec{e_1}^{(7)} \}, \\
 		& x^V_{[4,4]} = \{ \vec{e_3}^{(4)} - \half \vec{e_1}^{(4)} \}, \\
 		& \\
 		& x^{W'}_{[0,4]} =  \{ \vec{e_1}^{(0)}, \vec{e_1}^{(1)}, \vec{e_1}^{(2)}, \vec{e_1}^{(3)}, \vec{e_1}^{(4)} \}, \\
 		& x^{W'}_{[0,5]} =  \{ \vec{e_2}^{(0)} + \vec{e_1}^{(0)}, \vec{e_2}^{(1)} + \vec{e_1}^{(1)}, \vec{e_2}^{(2)} + \vec{e_1}^{(2)}, \vec{e_2}^{(3)} + \vec{e_1}^{(3)}, \vec{e_2}^{(4)} + \vec{e_1}^{(4)}, \vec{e_1}^{(5)} \}. \qedhere
 	\end{align*}
\end{example}

\subsection{Ladder Decompositions of an Interleaving Pair} \label{subsec:ladder_decomp_of_interleaving_pair}

	Let~$(\Phi, \Psi)$ be a~$\delta$-interleaving pair between modules~$V$ and~$W$. Representing the composition~$\Phi(\delta) \circ \Psi$ or~$\Psi(\delta) \circ \Phi$ in a single matrix in a chosen barcode basis gives
	\begin{align*} 
		\text{Id}_{\geq 2\delta}(x_{J_1}, x_{J_2}) =
		\begin{cases}
			1, & \text{if } |J_1| \geq 2\delta \text{ and } x_{J_1}(2\delta) = x_{J_2},\\
			0, & \text{otherwise,}
		\end{cases}
	\end{align*}
	which follows from the definition of a~$\delta$-interleaving pair. In the rest of this section we analyse these properties further to obtain results relating ladder decompositions of~$\Phi$ and~$\Psi$.
	\begin{remark} \label{rem:matrix_multiplication}
		When working with single matrix representations of morphisms of persistence modules as defined in~\cref{prop:shape_of_Phi_single_matrix}, we cannot rely on the intuition developed for matrices of linear maps between vector spaces. An example of such discrepancy is the fact that the matrix representation of the composition~$\Phi \circ \Psi$ is in general not equal to the matrix product of~$M_\Phi$ and~$M_\Psi$. However, it can be obtained from the matrix product as follows
		\begin{align*}
			M_{\Psi \circ \Phi}(x_{J_r}, x_{J_c}) = \begin{cases}
				(M_\Psi \cdot M_\Phi)(x_{J_r}, x_{J_c}), &  \text{if } J_r \preceq J_c,\\
				0, & \text{otherwise.}
			\end{cases}
		\end{align*}
		\begin{figure}[ht]
			\centering
			\begin{tikzpicture}
	\draw (0,0) -- (3,0) node[right, below]{$J_r$};
	\draw (2,1) -- (5,1) node[right, below]{$J'$};
	\draw (4,2) -- (7,2) node[right, below]{$J_c$};
\end{tikzpicture}
			\caption{Suppose we compose a morphism mapping~$x_{J'}$ to~$x_{J_r}$ with a morphism mapping~$x_{J_c}$ to~$x_{J'}$. Since the bars~$J_c$ and~$J_r$ do not intersect, the composition of these morphisms would not map between~$x_{J_c}$ and~$x_{J_r}$, which means the matrix representation of composition is not simply the product of matrix representations of morphisms.}
			\label{fig:empty_intersection}
		\end{figure}
		To clarify, the entry~$M_{\Psi \circ \Phi}(x_{J_r}, x_{J_c})$ might be zero if~$J_c$ and~$J_r$ have an empty intersection (see~\cref{fig:empty_intersection}).	
	\end{remark}
	Let~$i$ be the index of a row (or a column) and denote by~$x_i$ the bar generator corresponding to row (or column)~$i$. Denote the corresponding bar by~$J_i = [i_1, i_2]$.
	\begin{lemma} \label{l:second_morphism_matrix_repr}
		Let~$M_\Phi$ and~$M_\Psi$ be the matrix representations of morphisms~$\Phi\colon V \to W(\delta)$ and~${\Psi \colon W \to V(\delta)}$ making an interleaving pair in the barcode bases in which~$\Phi$ decomposes as in~\cref{thm:Xi_constant}. For any non-zero entry~${M_\Phi(r, c) = 1}$, the following hold:
	\begin{enumerate}
		\item If~$|J_c| \geq 2\delta$ or~$|J_r| \geq 2\delta$ then~$M_\Psi(c, r) = 1$.\label{1}
		\item If~$|J_c| \geq 2\delta$ and there exists~$z \neq c$ for which~$M_\Psi(z, r) \neq 0$ then
		\begin{enumerate}
			\item $J_z \preceq J_c$, \label{2b}
			\item $z_2 < c_1 + 2\delta$. \label{2c}
		\end{enumerate}
		\item If~$|J_r| \geq 2\delta$ and there exists~$z \neq r$ for which~$M_\Psi(c, z) \neq 0$ then
		\begin{enumerate}
			\item $J_r \preceq J_z$, \label{3b}
			\item $z_1 > r_2 - 2\delta$. \label{3c}
		\end{enumerate}
	\end{enumerate}
	\end{lemma}	
	\begin{proof}
		Suppose~$|J_c| \geq 2\delta$ and observe the matrix representation of the composition~$M_{\Psi \circ \Phi} = \text{Id}_{\geq 2\delta}$. Since the entry~$M_{\Psi \circ \Phi}(c, c)=1$ is non-zero, the~\cref{rem:matrix_multiplication} suggests that~$M_{\Psi \circ \Phi}(c, c) = (M_\Psi \cdot M_\Phi)(c, c).$ From the matrix equation
		\begin{equation*}
			M_\Psi \cdot M_\Phi \ \ =\ \  
			\begin{pNiceArray}{ccc:c:c}[first-row,first-col]
				&   &   &   & r &   \\
				&   &   &   &   &   \\
				\hdottedline
				c &   & * &   & a & * \\
				\hdottedline
				&   &   &   &   &   \\
				\hdottedline
				z &   & * &   & b & * \\
				\hdottedline
				&   &   &   &   &   
			\end{pNiceArray} \cdot 
			\begin{pNiceArray}{c:c:ccc}[first-row,last-col]
				& c &   &   &   &   \\
				&   &   &   &   &   \\
				& 0 &   &   &   &   \\
				&   &   &   &   &   \\
				\hdottedline
				0 & 1 &   & 0 &   & r \\
				\hdottedline
				& 0 &   &   &   &   
			\end{pNiceArray}
		\end{equation*}
	we can deduce that~$1 =(M_\Psi \cdot M_\Phi)(c, c) = 1 \cdot a = M_\Psi(c, r)$. With a similar procedure we can obtain the same result for when~$|J_r|\geq 2\delta$ which proves the first property.
	
	Continue with the assumption that~$|J_c| \geq 2\delta$ and there is a~$z$ such that~$M_\Psi(z, r) \neq 0$. Now the entry~$M_{\Psi \circ \Phi}(z, c)$ is zero, and since~$(M_\Psi \cdot M_\Phi)(z, c) = M_\Psi(z, r)$ is not zero, it must be that the bars~$J_z(2\delta)$ and~$J_c$ do not intersect. In other words,~$z_2 < c_1 + 2\delta$ (property~\ref{2c}). Now assume~$J_z \not\preceq J_c$, which means we are in one of the following cases:
	\begin{itemize}
		\item~$z_2 > c_2$,
		\item~$z_1 > c_1$,
		\item~$z_2 < c_1.$
	\end{itemize}
	The first case cannot happen since~$c_2 < z_2 < c_1 + 2\delta$ and~$|J_c|\geq 2\delta$ cannot hold simultaneously. In the second case, the chain of inequalities~$c_1 < z_1 \leq z_2 < c_1 + 2\delta \leq c_2$ gives us~$J_z \subset J_c$ and~$|z_1 - c_1| < 2 \delta$, which is a contradiction with the assumption that~$\ass$. Lastly, assume~$z_2 < c_1$. Since~$|J_c|\geq 2\delta$, the bar~$J_r$ must be the one from \cref{r:matches_of_long_bars}. In particular,
	\begin{align*}
		|c_i - r_i| \leq \delta \text{ for } i=1, 2.
	\end{align*}
	Since~$c_1 -\delta \leq r_1$, a bar~$J_z$ with the bar generator in~$\supp \Psi(x_r)$ must satisfy~$z_2 \geq r_1 + \delta \geq c_1$, which cannot be satisfied in the third case. We arrive to contradictions in all three cases, hence~$J_z \preceq J_c$ holds.
	
	Properties~\ref{3b} and~\ref{3c} can be obtained similarly by observing the entries of~$M_{\Phi \circ \Psi} = \text{Id}_{\geq 2\delta}$ and comparing them to
	\begin{equation*}
		M_\Phi \cdot M_\Psi \ \ =\ \ 
		\begin{pNiceArray}{c:c:ccc}[first-row,first-col]
			&   & c &   &   &   \\
			&   &   &   &   &   \\
			&   & 0 &   &   &   \\
			&   &   &   &   &   \\
			\hdottedline
			r & 0 & 1 &   & 0 &  \\
			\hdottedline
			&   & 0 &   &   &   
		\end{pNiceArray} \cdot
		\begin{pNiceArray}{c:c:c:c:c}[first-row,last-col]
			& z &   & r &   &   \\
			& * &   & * &   &   \\
			\hdottedline
			& a &   & b &   & c \\
			\hdottedline
			&   &   &   &   &   \\
			& * &   & * &   &  \\
			&   &   &   &   &   
		\end{pNiceArray}. \qedhere
	\end{equation*} 
\end{proof}

\begin{corollary}\label{cor:entries_in_matching_form}
	Let~$(\Phi, \Psi)$ be a~$\delta$-interleaving pair between~$V$ and~$W$ and let~$\ass$. Given a matching form of~$M_\Phi$, we can obtain such a matching form of~$M_\Psi$ that
	\begin{align*}
		M_\Phi(r, c) = M_\Psi(c, r)
	\end{align*}
	whenever~$|J_c| \geq 2\delta$ and~$|J_r| \geq 2\delta$.
\end{corollary}
\begin{proof}
	Write~$M_\Psi$ in the pair of barcode bases in which~$M_\Phi$ is in matching form and perform the reduction process using admissible operations~\ref{AO1},~\ref{AO2} and~\ref{AO1}. Since~$|J_c| \geq 2\delta$ and~$|J_r| \geq 2\delta$, the combined properties~\ref{2b} and~\ref{3b} from \cref{l:second_morphism_matrix_repr} guarantee that the row~$c$ to the left of~$M_\Psi(c, r)$ and the column~$r$ below~$M_\Psi(c, r)$ are zero at the start of the reduction. This means that the steps of the reduction before encountering~$M_\Psi(c, r)$ will not modify it. Thus it will stay unchanged until the end of the reduction.
\end{proof}

\correspondence
\begin{proof}
	For each appearance of~$\mathbf{R}_{J_W(\delta)}^{J_V}$ in the ladder decomposition of~$\Phi$ we have a unique entry~$M_{\Phi}(R, C) = 1$ in the reduced matrix in the matching form, where~$R$ is a row index belonging to the bar~$J_W(\delta)$ and~$C$ a column index belonging to the bar~$J_V$. By \cref{cor:entries_in_matching_form}, the entry~$M_{\Psi}(C, R)$ in the matching form of~$\Psi$ also equals~$1$. It corresponds to an appearance of~$\mathbf{R}_{J_V(\delta)}^{J_W}$ in the ladder decomposition of~$\Psi$. 
\end{proof}

\begin{example} \label{ex:running_ex_3}
	Continue~\cref{ex:running_ex_1,ex:running_ex_2} by considering the morphism~$\Psi$, which forms an interleaving with~$\Phi$. First, remember the ladder decomposition
	\begin{gather*}
		\mathbf{R}^{[0,4]}_{[0,4]} \oplus \mathbf{R}^{[1,7]}_{[0,5]} \oplus \mathbf{I}^+_{[4,4]} \qedhere
	\end{gather*}
	that we obtained for the~$1$-invertible morphism~$\Psi^{(1)}$ in \cref{ex:running_ex_2}. The ladder decomposition of~$\Phi$ is then
	\begin{gather*}
		\mathbf{R}^{[0,4]}_{[1,5]} \oplus \mathbf{R}^{[1,7]}_{[1,6]} \oplus \mathbf{I}^+_{[4,4]}. \qedhere
	\end{gather*}
	To obtain a ladder decomposition for~$\Psi$ as well, begin by writing the matrices~$\Psi_i$ in the (shifted equivalents of) barcode bases obtained in~\cref{ex:running_ex_2}:
	\begin{align*}
		&\begin{aligned}
			\Psi_{0,7} = 0,
		\end{aligned}
		&&\begin{aligned}
			\Psi_{1,2} = \begin{bmatrix} 1 & 0 \\ 0 & 1 \end{bmatrix},
		\end{aligned}
		&&\begin{aligned}
			\Psi_{3} = \begin{bmatrix} 1 & 0 \\ 0 & 1 \\ 0 & 0 \end{bmatrix},
		\end{aligned}
		&&\begin{aligned}
			\Psi_{4,5} = \begin{bmatrix} 0 & 1 \end{bmatrix},
		\end{aligned}
		&&\begin{aligned}
			\Phi_6 = \begin{bmatrix} 1 \end{bmatrix}.
		\end{aligned}
	\end{align*}
	Since the single matrix representation of~$\Psi$
	\begin{align}
		M_\Psi = \begin{bmatrix} 1 & 0 \\ 0 & 1 \\ 0 & 0 \end{bmatrix}
	\end{align}
	is already in matching form, further reduction is not necessary. The ladder decomposition of~$\Psi$ is
	\begin{gather*}
		\mathbf{R}^{[1,5]}_{[0,4]} \oplus \mathbf{R}^{[1,6]}_{[1,7]} \oplus \mathbf{I}^-_{[4,4]}. \qedhere
	\end{gather*}
\end{example}

\begin{remark}
	The ladder decompositions of~$\Phi$ and~$\Psi$ are in most cases obtained in different pairs of barcode bases~$(\mathcal{B}_V^\Phi, \mathcal{B}_W^\Phi)$ and~$(\mathcal{B}_V^\Psi, \mathcal{B}_W^\Psi)$ respectfully, which is not illustrated in~\cref{ex:running_ex_3}. This happens whenever~$M_\Psi$ written in bases~$(\mathcal{B}_V^\Phi, \mathcal{B}_W^\Phi)$ contains a non-zero entry~$M_\Psi(R, C')$ as
	\begin{equation*}
		\begin{pNiceArray}{cc:c:cc:c:cc}[first-row,last-col]
			&   & C &   &   & C' &   &   &   \\
			&   &   &   &   &    &   &   &   \\
			\hdottedline
			&   & 1 &   &   & * &   &   & R \\
			\hdottedline
			&   &   &   &   &   &   &   &   \\
			\hdottedline
			&   &   &   &   & 1 &   &   & R' \\
			\hdottedline
			&   &   &   &   &   &   &   &   
		\end{pNiceArray}
	\end{equation*}
	or when a similar entry can be found in~$M_\Phi$ written in bases~$(\mathcal{B}_V^\Psi, \mathcal{B}_W^\Psi)$.
\end{remark}

	\section{\texorpdfstring{$q$}\ -Coarse Ladder Decomposition} \label{sec:quasi-barcode_form}

As deduced in~\cref{subsec:decomp_of_interleavings}, the range of parameters~$\delta$ for which~$\delta$-invertible morphisms between persistence modules~$V$ and~$W$ will decompose as in~\cref{thm:Xi_constant} is controlled by the nestedness of~$V$ and~$W$. More precisely, small nestedness imposes a harsher limit on parameter~$\delta$ in \cref{thm:Xi_constant}. In this section we focus on persistence modules with small nestedness which is achieved in (relatively) short bars, as illustrated in~\cref{fig:long_bar_with_short_nested_bars}. We explore to which extent short nested bars can be ignored and how this weakens our results from~\cref{subsec:decomp_of_interleavings,subsec:ladder_decomp_of_interleaving_pair}.
\begin{figure}[ht]
	\centering
	\begin{tikzpicture}
	\draw (0,-0.75) -- (3,-0.75);
	\draw (-0.4,-1) -- (0.4,-1);
	\draw (-0.3,-1.25) -- (0.2,-1.25);
\end{tikzpicture}
	\caption{Example of a barcode of a persistence module with small nestedness. The minimum from the definition of nestedness is achieved in the lower two bars, which are significantly shorter than the top-most bar and can in some cases be attributed to noise.}
	\label{fig:long_bar_with_short_nested_bars}
\end{figure}
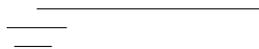
 \begin{definition}
	Let~$V$ be a persistence module and~$q \in \R_{\geq 0}$. An isomorphism~$V \cong V_{\geq q} \oplus V_{< q}$, where the pair~$(V_{\geq q}, V_{< q})$ of persistence modules satisfies
	\begin{itemize}
		\item~any bar $J \in \Barc{V_{\geq q}}$ is of length at least~$q$,
		\item~any bar $J \in \Barc{V_{< q}}$ is of length smaller than~$q$,
	\end{itemize}
	is called a~\emph{$q$-splitting} of~$V$.
\end{definition}
A~$q$-splitting induces epimorphisms~$pr_{\geq q}^V, pr_{< q}^V$ and monomorphisms~$i_{\geq q}^V, i_{< q}^V$, which map as follows:
\begin{center}
	\begin{tikzcd}
		V_{\geq q} \arrow[rr, "i_{\geq q}^V", hook, shift left] & & V \arrow[rr, "pr_{< q}^V", two heads, shift left] \arrow[ll, "pr_{\geq q}^V", two heads, shift left] & & V_{< q} \arrow[ll, "i_{< q}^V", hook, shift left].
	\end{tikzcd}
\end{center}
Whenever we wish to discard shorter bars, let us say shorter than~$q$, we will project the module~$V$ onto the part~$V_{\geq q}$ of its splitting. Refer to~$V_{\geq q}$ as the~\emph{$q$-coarse part} of~$V$. A~$q$-splitting of a persistence module always exists and is especially convenient since it allows us to define barcode bases~$\mathcal{B}_{V_{\geq q}}$ and~$\mathcal{B}_{V_{< q}}$ for~$V_{\geq q}$ and~$V_{< q}$ separately. The barcode basis~$\mathcal{B}_V$ these induce on~$V$ is then defined simply as
\begin{align*}
	\mathcal{B}_V = i_\q^V(\mathcal{B}_{V_{\geq q}}) \cup i_{< q}^V(\mathcal{B}_{V_{< q}}).
\end{align*}
%

Denote by~$[ \delta ]_V$ the collection of inner morphisms~$\{v_{i, i+\delta}\}_i$ of module~$V$. For brevity sake we often omit the subscript denoting the module and write only~$[\delta]$.
\begin{lemma} \label{l:q-coarse_properties}
	For any parameter~$q$ and a persistence module~$V$ the following properties hold:
	\begin{enumerate}
		\item $ \displaystyle \Xi(V_{\geq q}) = \min_{\substack{[c, d] \subset [a, b] \in \Barc{M} \\ d-c \geq q}} \min \{ |a-c|, |b-d| \}.$ \label{l:prop1}
		\item $d_B(\Barc{V}, \Barc{V_{\geq q}}) < \qhalf$. \label{l:prop2}
		\item $V$ and~$V_{\geq q}$ are~$\qhalf$-interleaved with interleaving morphisms \label{l:prop3}
		\begin{align*}
			\varphi_\q^V\colon V \to V_\q (\qhalf) \qquad & \qquad \tilde{\varphi}_{\geq q}^V\colon V_\q \to V(\qhalf) \\
			\varphi_\q^V = [\qhalf] \circ pr_\q^V \qquad  & \qquad \tilde{\varphi}_{\geq q}^V = i_\q^V(\qhalf) \circ [\qhalf].
		\end{align*}
		\item Epimorphism~$pr_{\geq q}^V$ and monomorphism~$i_{\geq q}^V$ are~$\qhalf$-invertible. \label{l:prop4}
	\end{enumerate}
\end{lemma}
\begin{proof}
	Property~\ref{l:prop1} is rather obvious since the barcode~$\Barc{V_\q}$ can be obtained from~$\Barc{V}$ by discarding bars shorter than~$q$. To prove Property~\ref{l:prop2} observe that the the bottleneck distance will be achieved in a matching where all bars in~$\Barc{V}$ of length at least~$q$ are matched with their copies in~$\Barc{V_\q}$ and the rest are left unmatched. The first contribute nothing to the cost, while the second are of length~$< q$, and not matching them contributes less than~$\qhalf$ to the cost. As a consequence of Property~\ref{l:prop2} and the algebraic stability theorem~\cite{bauer2013induced}, modules~$V$ and~$V_\q$ are~$\qhalf$-interleaved. To prove that a possible choice for~$\qhalf$-interleaving morphisms are~$\varphi_\q^V$ and~$\tilde{\varphi}_{\geq q}^V$, see that
	\begin{align*}
		\tilde{\varphi}_{\geq q}^V(\qhalf) \circ \varphi_\q^V = i_\q^V(q) \circ [q] \circ pr_\q^V= [q] \circ i_\q^V \circ pr_\q^V,
	\end{align*}
	where we use the morphism property~$i_\q^V(q) \circ [q] = [q] \circ i_\q^V$ in the last step. Further, since~$V_{< q} \subseteq \ker ([q])$ we obtain the desired
	\begin{align*}
		\tilde{\varphi}_{\geq q}^V(\qhalf) \circ \varphi_\q^V = [q].
	\end{align*}
	The proof that~$\varphi_\q^V(\qhalf) \circ \tilde{\varphi}_{\geq q}^V = [q]$ also follows from similar considerations. These equalities together finish the proof of Property~\ref{l:prop3}. To prove that~$pr_{\geq q}^V$ is~$\qhalf$-invertible, we simply need to restate Property~\ref{l:prop3} as
	\begin{align*}
		& (i_\q^V(q) \circ [q]) \circ pr_\q^V = [q]\\
		& pr_\q^V(q) \circ (i_\q^V(q) \circ [q]) = pr_\q^V(q) \circ ([g] \circ i_\q^V) = [q]
	\end{align*}
	and see that~$i_\q^V(q) \circ [q]$ is its~$\qhalf$-inverse. Similarly,~$i_\q^V$ is~$\qhalf$-invertible with~$pr_\q^V(q) \circ [q]$ as its~$\qhalf$-inverse.
\end{proof}

\subsection{\texorpdfstring{$q$}\ -Coarse Ladder Decomposition of a~\texorpdfstring{$\delta$}\ -invertible Morphism} \label{subsec:qcoarse_invertible}
In this section we state the weaker versions of our results from \cref{sec:morphisms_and_induced_matchings} that hold for~$\delta$-invertible morphisms.

\begin{lemma} \label{c:induced_d'_invertible_morphism}
	Let~$\Phi \colon V \to W$ be a~$\delta$-invertible morphism with~$\delta$-inverse~$\Psi$. They induce:
	\begin{enumerate}
		\item a~$(\delta + \qhalf)$-invertible morphism~$pr_\q^W \circ \Phi \colon V \to W_\q$ with a~$(\delta + \qhalf)$-inverse~$\Psi(q) \circ i_\q^W(q) \circ [q]$,
		\item a~$(\delta + \qhalf)$-invertible morphism~$\Phi \circ i_\q^V \colon V_\q \to W$ with a~$(\delta + \qhalf)$-inverse~$pr_\q^V(q+2\delta) \circ [q](2\delta) \circ \Psi$,
		\item a~$(\delta + \qhalf)$-invertible morphism~$\tilde{\Phi}\colon V_\q \to W_\q$ with a~$(\delta + \qhalf)$-inverse~$\tilde{\Psi}$ where
		\begin{align*}
			& \tilde{\Phi} = pr_\q^W \circ \Phi \circ i_\q^V, \\
			&\tilde{\Psi} = pr_\q^V(2\delta+q) \circ \Psi(q) \circ (i_\q^W(q) \circ [q]).
		\end{align*}
	\end{enumerate}
	\begin{center}
	\begin{tikzcd}[scale=0.9]
		V \arrow[r, "\Phi"] \arrow[rd, "pr_\q^W \circ \Phi"', dashed] & W \arrow[d, "pr_\q^W"] &&
		V \arrow[r, "\Phi"] & W &&
		V \arrow[r, "\Phi"] & W \arrow[d, "pr_\q^W", two heads] \\
		
		& W_\q &&
		V_\q \arrow[u, hook, "i_\q^V"] \arrow[ru, "\Phi \circ i_\q^V"', dashed]& &&
		V_\q \arrow[u, hook, "i^V_\q"] \arrow[r, dashed, "\tilde{\Phi}"] & W_\q
	\end{tikzcd}
\end{center}
\end{lemma}
\begin{proof}
	The first two statements are a simple consequence of the fact that a composition of a~$\delta$-invertible morphism with a~$\qhalf$-invertible morphism is a~$(\delta + \qhalf)$-invertible morphism. To show the third, draw the composition~$\tilde{\Psi} \circ \tilde{\Phi}$ in a diagram as
	\begin{center}
		\begin{tikzcd}
			V \arrow[r, "\Phi"]  \arrow[ddr, dashed] & W \arrow[dd, "pr_\q^W"', two heads] & & W(q) \arrow[r, "\Psi(q)"] & V(2\delta+ q) \arrow[dd, "{pr_\q^V(2\delta + \qhalf)}", two heads] \\
			&&&&\\
			V_\q \arrow[uu, "{i_\q^V}", hook] \arrow[r, "\tilde{\Phi}"'] & W_\q  \arrow[uurrr, dashed] \arrow[rrr, "\tilde{\Psi}"'] \arrow[rruu, "i_\q^W(q) \circ {[q]}" description, hook] & & & V_\q (2\delta + q).
		\end{tikzcd}
	\end{center}
	where the dashed arrows denote compositions~$pr_\q^V \circ \Phi$ of a~$\delta$- and~$\qhalf$-invertible morphism, and~$\Psi(q) \circ i_\q^W(q) \circ [q]$ of their~$\delta$- and~$\qhalf$-inverses. Consequently, the composition
	\begin{align*}
		\Psi(q) \circ i_\q^W(q) \circ [q]  \circ pr_\q^V \circ \Phi
	\end{align*}
	is simply the inner morphism~$[2\delta + q]_V.$ Further,
	\begin{align*}
		\tilde{\Psi} \circ \tilde{\Phi} = pr_\q^V(2\delta + q) \circ [2\delta + q]_V \circ i_\q^V
	\end{align*}
	is simply the inner morphism~$[2\delta + q]_{V_\q}$. The proof that~$\tilde{\Phi}(2\delta+q) \circ \tilde{\Psi} = [2\delta + q]_{W_\q}$ follows in a similar way. Draw the composition~$\tilde{\Phi}(2\delta+q) \circ \tilde{\Psi}$ in a diagram as
	\begin{center}
		\begin{tikzcd}
			W \arrow[r, "\Psi"] \arrow[ddrrr, dashed] & V(2\delta) \arrow[ddrr, "pr_\q^V(2\delta+q) \circ {[q]}" description, two heads] & & V(2\delta+ q) \arrow[r, "\Phi(2\delta+q)"] & W(2\delta +q) \arrow[dd, "pr_\q^W(2\delta+q)", two heads, shift left] \\
			&&&& \\
			W_\q \arrow[rrr, "\tilde{\Psi}"'] \arrow[uu, "i_\q^W", hook] & & & V_\q (2\delta + q) \arrow[uu, "{i_\q^V (2\delta+q)}"', hook, shift left] \arrow[r, "\tilde{\Phi}(2\delta+q)"'] \arrow[uur, dashed] & W_\q(2\delta+q),
		\end{tikzcd}
	\end{center}
	where we use the fact that morphisms commute with the inner morphisms~$[q]$ to equivalently write
	\begin{align*}
		\tilde{\Psi} = pr_\q^V(2\delta+q) \circ \Psi(q) \circ (i_\q^W(q) \circ [q]) \qquad \text{ as } \qquad (pr_\q^V(2\delta+q) \circ [q]) \circ \Psi \circ i_\q^W.
	\end{align*}
	The dashed arrows again denote compositions~$(\Phi \circ i_\q^V)(2\delta+q)$ of a~$\delta$- and~$\qhalf$-invertible morphism (on the right), and~$pr_\q^V(2\delta+q) \circ [q] \circ \Psi$ of their~$\delta$- and~$\qhalf$-inverses (on the left). The composition of the dashed arrows is therefore the family of inner morphisms~$[2\delta +q]_W$ and the composition
	\begin{align*}
		pr_\q^W(2\delta+q) \circ [2\delta + q]_W \circ i_\q^W 
	\end{align*}
	is also a family of inner morphisms,~$[2\delta+q]_{W_\q}.$
\end{proof}

The following theorem lists the conditions that must be satisfied so that the induced morphisms from \cref{c:induced_d'_invertible_morphism} decompose as ladder persistence modules.
\qmain

\begin{proof}
	All three statements of this proposition are a direct consequence of~\cref{thm:Xi_constant}. Let us provide details only for the proof of the last one, since we follow the same approach in all cases.
	
	The morphism~$\tilde{\Phi} = pr_\q^W \circ \Phi \circ i_\q^V$ is a~$(\delta + \qhalf)$-invertible morphism between~$V_\q$ and~$W_\q$ by~\cref{c:induced_d'_invertible_morphism}. Since
	\begin{align*}
		\delta + \qhalf < \half \min(\Xi(V_\q), \Xi(W_\q))
	\end{align*}
	holds by our assumption, we can apply~\cref{thm:Xi_constant} to~$\tilde{\Phi}$. This means there is a pair of barcode bases~$\mathcal{B}_{V_\q}$ and~$\mathcal{B}_{W_\q}$ in which~$\Phi'$ decomposes as a ladder persistence module. As noted before, the fact that~$V_\q$ is the~$q$-coarse part of a $q$-splitting of~$V$ means that~$\mathcal{B}_{V_\q}$ can be supplemented with any barcode basis~$\mathcal{B}_{V_{<q}}$ to form a barcode basis for~$V$. By extending both~$\mathcal{B}_{V_\q}$ and~$\mathcal{B}_{W_\q}$ we obtain barcode bases~$\mathcal{B}_V$ and~$\mathcal{B}_W$ claimed to exist by the theorem.
\end{proof}

\subsection{\texorpdfstring{$q$}\ -Coarse Ladder Decompositions of a \texorpdfstring{$\delta$}\ -Interleaving Pair} \label{subsec:qcoarse_interleaving}

As before, we can leverage the correspondence between~$\delta$-invertible morphisms and~$\delta$-interleavings to obtain a statement similar to \cref{prop:q-decompositions_invertible} that holds for a single morphism in a~$\delta$-interleaving pair. We state the analogues of the two theoretical results of \cref{subsec:qcoarse_invertible} here, omitting all the proofs, since they are simple exercises in applying the correspondence from \cref{r:interleaving_vs_invertible}. Comparing the~$q$-coarse ladder decompositions of the two morphisms making a~$\delta$-interleaving we again show that there is a nice correspondence between them for all bars of sufficient length.

\begin{lemma} [Analogue of \cref{c:induced_d'_invertible_morphism}] \label{l:induced_d'_interleavings}
	Let~$(\Phi, \Psi)$ be a~$\delta$-interleaving pair between modules~$V$ and~$W$. They induce morphisms:
	\begin{enumerate}
		\item $\varphi_\q^W(\delta) \circ \Phi \colon V \to W_\q(\delta+\qhalf)$ and~$\Psi(\qhalf) \circ \tilde{\varphi}_\q^W \colon W_\q \to V(\delta+\qhalf)$ making a~$(\delta + \qhalf)$-interleaving pair, \label{int1}
		\item $\Phi(\qhalf) \circ \tilde{\varphi}_\q^V \colon V_\q \to W(\delta+\qhalf)$ and~$\varphi_\q^V(\delta) \circ \Psi \colon W \to V_\q(\delta+\qhalf)$ making a~$(\delta + \qhalf)$-interleaving pair,\label{int2}
		\item $\tilde{\Phi}\colon V_\q \to W_\q(\delta+\qhalf)$ and~$\tilde{\Psi}:W_\q \to V_\q(\delta+\qhalf)$ where
		\begin{align*}
			&\tilde{\Phi} = \varphi_\q^W(\delta) \circ \Phi \circ i_\q^V, \\
			&\tilde{\Psi} = \varphi_\q^V(\delta) \circ \Psi \circ i_\q^W,
		\end{align*}
		which make a~$(\delta + \qhalf)$-interleaving pair. \label{int3}
	\end{enumerate}
\end{lemma}
\begin{center}
\begin{tikzcd}[scale=0.9]
	V \arrow[r, "\Phi"] \arrow[rd, "\varphi_\q^W(\delta) \circ \Phi"', dashed] & W(\delta) \arrow[d, "\varphi_\q^W(\delta)"] &&
	V(\qhalf) \arrow[r, "\Phi(\qhalf)"] & W(\delta+\qhalf)  &&
	V \arrow[dr, "\varphi_\q^W(\delta) \circ \Phi"] & \\
	
	& W_\q(\delta+\qhalf) &&
	V_\q \arrow[u, "\tilde{\varphi}^V_\q"] \arrow[ru, "\Phi(\qhalf) \circ \tilde{\varphi}_{\geq q}^V"', dashed]& &&
	V_\q \arrow[u, hook, "i^V_\q"] \arrow[r, dashed, "\tilde{\Phi}"] \arrow[rd, "\Phi(\qhalf) \circ \tilde{\varphi}_{\geq q}^V"',] & W_\q(\delta+\qhalf)\\
	
	& && & && & W(\delta+\qhalf) \arrow[u, two heads, "pr_\q^W"']
\end{tikzcd}
\end{center}

\begin{theorem}[Analogue of \cref{prop:q-decompositions_invertible}] \label{prop:q-decompositions}
	Let~$\Phi\colon V \to W(\delta)$ be a morphism which is part of a~$\delta$-interleaving pair. If there exists a parameter~$q$ such that
	\begin{enumerate}
		\item $\delta < \half \min(\Xi(V), \Xi(W_{\geq q})) - \qhalf$, then the induced morphism~$\varphi_\q^W(\delta) \circ \Phi$ decomposes as a ladder persistence module. \label{interleaving1}
		\item $\delta < \half \min(\Xi(V_{\geq q}), \Xi(W)) - \qhalf$, then the induced morphism~$\Phi(\qhalf) \circ \tilde{\varphi}_{\geq q}^V$ decomposes as a ladder persistence module. \label{interleaving2}
		\item $\delta < \half \min(\Xi(V_{\geq q}), \Xi(W_{\geq q})) - \qhalf$, then the induced morphism~$\tilde{\Phi}$ from~\cref{l:induced_d'_interleavings} decomposes as a ladder persistence module. \label{interleaving3}
	\end{enumerate}
	All these decompositions are obtained in partial barcode bases that can be extended to barcode bases of modules~$V$ and~$W$.
\end{theorem}

Notice how \cref{l:induced_d'_interleavings} and \cref{prop:q-decompositions} fit together. Given a~$\delta$-interleaving pair~$(\Phi, \Psi)$, morphism~$\Phi$ satisfies the assumptions of case~\eqref{interleaving1} of \cref{prop:q-decompositions} for some~$q$ if and only if the morphism~$\Psi$ satisfies the assumptions of case~\eqref{interleaving2} of \cref{prop:q-decompositions} for the same parameter~$q$. The induced morphisms each of them gives make a~$(\delta + \qhalf)$-interleaving pair by \cref{l:induced_d'_interleavings}\eqref{int1}. By switching the roles of~$\Phi$ and~$\Psi$ we can see the reverse also holds giving us the $(\delta + \qhalf)$-interleaving pair from \cref{l:induced_d'_interleavings}\eqref{int2}. Further, morphisms~$\Phi$ and~$\Psi$ satisfy the assumptions of case~\eqref{interleaving3} simultaneously, giving us the $(\delta + \qhalf)$-interleaving pair from \cref{l:induced_d'_interleavings}\eqref{int3}. This means \cref{prop:q-decompositions} assures that whenever one of the morphisms in the pair can be decomposed, then so can the other. More importantly, we can compare them.

Let~$(\Phi', \Psi')$ be any of the induced~$(\delta +\qhalf)$-interleaving pairs from \cref{l:induced_d'_interleavings} for which the ladder decomposition can be obtained by \cref{prop:q-decompositions}. By applying \cref{cor:entries_in_matching_form} and \cref{thm:similarity_of_bases} to~$(\Phi', \Psi')$ we obtain the following result.

\begin{corollary}\label{cor:same_entries_qquasi}
	Given a matching form of~$M_{\Phi'}$, we can obtain such a matching form of~$M_{\Psi'}$ that
	\begin{align*}
		M_\Phi'(r, c) = M_\Psi'(c, r)
	\end{align*}
	whenever~$|J_c| \geq 2\delta + q$ and~$|J_r| \geq 2\delta+q$. For such two bars and any~$\mu \in \mathbb{N}$, the following statements are equivalent
	\begin{itemize}
		\item $(\mathbf{R}_{J_r}^{J_c})^\mu$ appears in the ladder decomposition of~$\Phi'$,
		\item $(\mathbf{R}_{J_c}^{J_r})^\mu$ appears in the ladder decomposition of~$\Psi'$.
	\end{itemize}
\end{corollary}
Since the barcode bases in which we write the matrix~$M_\Phi'$ can be expanded to barcode bases of the whole persistence modules in the domain and codomain of~$\Phi$, we can think of~$M_\Phi'$ as a sub-matrix of~$M_\Phi$.

	\section{Induced Partial Matchings} \label{sec:induced_partial_matchings}

Here we give a brief introduction to multisets, which barcodes are examples of. A \emph{multiset}~$S$ is a set where each element~$s \in S$ has a non-zero multiplicity~$\mu(s) \in \Z_{>0}$. An isomorphism of multisets is a bijection of the underlying sets that preserves the multiplicities. A \emph{sub-multiset} is a subset~$T \subseteq S$ in which the multiplicity of each element is not bigger than its multiplicity in~$S$. A morphism of multisets~$S_1$ and~$S_2$ consists of sub-multisets~$T_1 \subseteq S_1$ and~$T_2 \subseteq S_2$ and an isomophism~$\chi: T_1 \to T_2$. We call it a \emph{partial matching} between~$S_1$ and~$S_2$ and denote it by 
\begin{align*}
	\chi \colon S_1 \todot S_2.
\end{align*}
We often write and define partial matchings as multisets of pairs~$(t_1, \chi(t_1)) \in T_1 \times T_2$. They appear in the definitions of various notions of distances on the space of barcodes, such as the bottleneck~\cite{bottleneck_distance} and the $p$-Wasserstein distance~\cite{wasserstein_distance}. More precisely, each of these distances~$d(B_1, B_2)$ is the smallest cost of a partial matching between~$B_1$ and~$B_2$, where the associated cost is different for each of the distances. In the case of the bottleneck distance, which is relevant for the use in this paper, the \emph{cost of a partial matching}~$\chi \colon S_1 \todot S_2$ is
\begin{align*}
	\max \Big\{  \max_{(I, J) \in \chi} \{ \max (|i_1 - j_1|, |i_2 - j_2|)\}, \max_{\substack{J \in B_1 \cup B_2 \\ J \notin \chi}} \half(j_2 - j_1)  \Big\},
\end{align*}
where~$I=[i_1, i_2]$ and~$J=[j_1, j_2]$.

Given a morphism between one-parameter persistence modules, one might ask whether it induces a partial matching on the level of barcodes. This question was central in the proof of \emph{algebraic stability theorem}~\cite{bauer2013induced}, when a~\emph{BL induced matching} (BL stands for ``Bauer and Lesnick'') was introduced. In this section we look at an alternative construction of a morphism induced partial matching given by the ladder decompositions. We compare them to the BL-induced matching and show that they are an example of \emph{basis-independent induced matchings}~\cite{basis_independent_matchings}. We conclude the paper with the matchings induced by the ladder decompositions of the coarser versions of the~$\delta$-invertible morphisms.

\subsection{Ladder Decomposition Induced Partial Matching}

In~\cite{jacquard2021space}, Jacquard et al.\ observe that an alternative definition of an induced partial matching can be retrieved from the ladder decomposition of the morphism in question.
\begin{corollary}[of \cref{thm:jacquard_induced_matching}] \label{cor:matching_induced_by_decomp}
	Ladder decomposition of morphism~$\Phi \colon V \to W$ induces a matching of barcodes~$\Barc{V}$ and~$\Barc{W}$ defined as
	\begin{align*}
		&\chi_\Phi \colon \Barc{V} \todot \Barc{W} \\
		&\chi_\Phi = \{ ((J_1, J_2), r_{J_1}^{J_2}) \mid \mathbf{R}_{J_1 }^{J_2} \text{ appears in ladder decomposition of }\Phi \}.
	\end{align*}
	Note that~$\chi_\Phi$ is a multiset in which each pair~$(J_1, J_2)$ appears with the multiplicity~$r_{J_1}^{J_2}$.
\end{corollary}

The uniqueness of the ladder decomposition implies that the induced matching is also unique. \cref{thm:Xi_constant} assures the existence of ladder decomposition for a wider range of~$\delta$-invertible morphisms, which induce partial matchings in a similar way. By the correspondence from \cref{r:interleaving_vs_invertible} the same holds for morphisms that are part of a~$\delta$-interleaving pair. The following results describe the properties of the matchings they induce.
\cost
\begin{proof}
	The decomposition
	\begin{align*}
		(V, W, \Phi) \cong \bigoplus_{[i_1, j_1] \preceq [i_2, j_2]}\Big( \mathbf{R}_{[i_1, j_1]}^{[i_2, j_2]} \Big)^{r_{[i_1, j_1]}^{[i_2, j_2]}} \oplus \bigoplus_{i \leq j} \Big( \mathbf{I}^+ [i_1, j_1] \Big)^{d_{ij}^+} \oplus \bigoplus_{i \leq j} \Big( \mathbf{I}^- [i_1, j_1] \Big)^{d_{ij}^-}
	\end{align*}
	is obtained by finding a pair of barcode bases~$(\mathcal{B}_V, \mathcal{B}_W)$ in which the matrix representation of~$\Phi$ is in matching form. Remember, each appearance of~$\mathbf{R}_{J_r}^{J_c}$ in the decomposition corresponds to a non-zero entry in the matrix~$M_\Phi$ in a row~$r$ belonging to~$J_r$ and column~$c$ belonging to column~$J_c$.
	Similarly, each appearance of~$\mathbf{I}^+_J$ corresponds to an empty column in~$M_\Phi$ belonging to bar~$J$, and each appearance of~$\mathbf{I}^-_J$ corresponds to an empty row in~$M_\Phi$ belonging to bar~$J$.
	
	First, let us prove that the matched bars contribute a cost smaller than $\delta$. If the bars~$[i_2, j_2] \in \Barc{V}$ and~$[i_1, j_1] \in \Barc{W}$ are matched, a generator of bar~$x_{[i_1, i_1]}$ is in the support of the image of the generator~$x_{[i_2, j_2]}$ with the $\delta$-invertible morphism~${\Phi}$. By \cref{l:existence_of_matched_bar} this implies that~$|i_2 - i_1| \leq \delta$ and~$|j_2-j_1| \leq \delta$. Therefore, the cost of matching these bars is smaller or equal to~$\delta.$
	
	All there is left to prove is that the cost the bars that are left unmatched contribute is less than~$\delta$ as well. Assume~$\mathbf{I}^+[i, j]$ for~$[i, j]\in \Barc{V}$ appears as a summand in the decomposition. Since the~$\supp \Phi(x_{[i, j]})$ is empty, \cref{l:existence_of_matched_bar} implies that $|j-i| < 2\delta$ and the cost of not matching it is smaller than~$\delta.$ In a similar manner one can show that if~$\mathbf{I}^-[i, j]$ appears as a summand in the decomposition, then the cost of not matching~$[i, j]\in \Barc{W}$ is smaller than~$\delta$.
\end{proof}
\matchings
\begin{proof}
	By \cref{thm:similarity_of_bases},~$M_\Phi(r, c) = 1$ implies~$M_\Phi(c, r) = 1$ for index~$c$ corresponding to the bar~$J_V$ and~$r$ corresponding to bar~$J_W$. Consequently the multiplicity~$r_{J_W}^{J_V}$ of~$\mathbf{R}_{J_W}^{J_V}$ in the ladder decomposition of~$\Phi$ is smaller or equal to the multiplicity~$r_{J_V}^{J_W}$ of~$\mathbf{R}_{J_W}^{J_V}$ in the ladder decomposition of~$\Psi$. By using the same arguments with the roles of~$\Phi$ and~$\Psi$ reversed, we obtain that $r_{J_V}^{J_W} \leq r_{J_W}^{J_V}$. Considering the definition of the ladder decomposition induced partial matching, this concludes the proof.
\end{proof}
\begin{example}\label{ex:running_ex_4}
	Return to the~$\delta$-interleaving pair~$(\Phi, \Psi)$ from \cref{ex:running_ex_1,ex:running_ex_2,ex:running_ex_3}. The ladder decompositions we obtained are
	\begin{gather*}
		(V, W(\delta), \Phi) \cong \mathbf{R}^{[0,4]}_{[1,5]} \oplus \mathbf{R}^{[1,7]}_{[1,6]} \oplus \mathbf{I}^+_{[4,4]} \qquad \text{and} \qquad (W, V(\delta), \Psi) \cong \mathbf{R}^{[1,5]}_{[0,4]} \oplus \mathbf{R}^{[1,6]}_{[1,7]} \oplus \mathbf{I}^-_{[4,4]}
	\end{gather*}
	respectively. The matchings they induce are therefore
	\begin{align*}
		\chi_\Phi &= \{ ([0,4], [1,5]), \ ([1,7], [1,6]) \}\\
		\chi_\Psi &= \{ ([1,5], [0,4]), \ ([1,6], [1,7]) \}.
	\end{align*}
	They happen to be the opposite matchings, which is not always the case (they can differ on bars shorter than~$2\delta$). It is easy to see that they are of cost~$1$, which agrees with \cref{cor:cost_of_induced_matching}.
\end{example}

\subsection{Comparisson with the Bauer-Lesnick Induced Matchings} \label{subsubsec:BL}
To our knowledge the first notion of a partial matching induced by a morphism of persistence modules was introduced by Bauer and Lesnick in~\cite{bauer2013induced, bauer2020persistence}. Requiring a choice of an order on bars with the same endpoints, the construction follows three steps:
\begin{enumerate}
	\item A (general) morphism~$\Phi \colon V \to W$ is split into a surjection onto its image and inclusion into the codomain as follows:
	\begin{align*}
		V\ \xrightarrowdbl{q_\Phi}\  \text{im}\ \Phi\ \xhookrightarrow{i_\Phi}\ W.
	\end{align*}
	The matchings~$\chi^{BL}_{q_\Phi}$ and~$\chi^{BL}_{i_\Phi}$ are defined separately and later combined into a single matching~$\chi^{BL}_\Phi$ as
	\begin{align*}
		\chi^{BL}_\Phi = \{ (J_V, J_W) \mid \exists J_{\text{im}}\in \Barc{\text{im} \Phi} \text{ s.t.\ } (J_V, J_{\text{im}}) \in \chi^{BL}_{q_\Phi} \text{ and } (J_{\text{im}}, J_W) \in \chi^{BL}_{i_\Phi} \}.
	\end{align*}
	\item The matching of the injection~$i_\Phi$ is constructed for each family~$\langle \cdot , d \rangle$ of bars with the second endpoint~$d$ individually. First, both~$\langle \cdot , d \rangle_{\Barc{\text{im} \Phi}}$ and~$\langle \cdot , d \rangle_{\Barc{W}}$ are ordered by the length decreasingly, combining it with the chosen order on bars of the same length. Then the~$n$-th bar in~$\langle \cdot , d \rangle_{\Barc{\text{im} \Phi}}$ gets matched with the~$n$-th bar in~$\langle \cdot , d \rangle_{\Barc{W}}$. If the cardinalities differ, the residual bars are left unmatched. The matchings of families~$\langle \cdot , d \rangle$ for all possible endpoints~$d$ are combined into a matching~$\chi^{BL}_{i_\Phi}$.
	\item The matching of the surjection~$q_\Phi$ is constructed for each family~$\langle b, \cdot \rangle$ of bars with the first endpoint~$b$ individually. As before, families~$\langle b, \cdot \rangle_{\Barc{\text{im} \Phi}}$ and~$\langle b, \cdot \rangle_{\Barc{W}}$ are ordered as before and bars get matched based on their position in the order. The matchings of families~$\langle b, \cdot \rangle$ for all possible endpoints~$b$ are combined into a matching~$\chi^{BL}_{q_\Phi}$.
\end{enumerate}
As noted by the authors, the construction is determined by the barcodes~$\Barc{V}$,~$\Barc{W}$ and~$\Barc{\text{im} \Phi}.$ This means that the only way the morphism influences the construction is not through its image, but through the barcode~$\Barc{\text{im} \Phi}$. More explicitly, as long as the barcodes of the image of two parallel morphisms are the same, the induced matchings will coincide. Let us illustrate this with an example.

\begin{example} \label{ex:BL}
	Let~$V$ and~$W$ be the following persistence modules:
	\begin{center}
		\begin{tikzcd}
			V\colon & 0 \arrow[r, yshift=0.7ex]
			& \mathbb{F}^2 \arrow[r, yshift=0.7ex, "\text{Id}"]
			& \mathbb{F}^2 \arrow[r, yshift=0.7ex, "\text{Id}"]
			& \mathbb{F}^2 \arrow[r, yshift=0.7ex, "\pprojfirst"]
			& \mathbb{F} \arrow[r, yshift=0.7ex]
			& 0, \\
			W\colon & 0 \arrow[r, yshift=0.7ex]
			& \mathbb{F} \arrow[r, yshift=0.7ex, "\pinjfirst"]
			& \mathbb{F}^2 \arrow[r, yshift=0.7ex, "\text{Id}"]
			& \mathbb{F}^2 \arrow[r, yshift=0.7ex, " \text{Id} "]
			& \mathbb{F}^2 \arrow[r, yshift=0.7ex]
			& 0.
		\end{tikzcd}
	\end{center}
	Notice that we assume a specific choice of barcode bases in which we have written the transition maps. The bottleneck and interleaving distances between these modules are both~$1$ and there is an obvious choice for a~$1$-invertible morphism, namely
	\begin{center}
		\begin{tikzcd}
			V\colon \arrow[dd, "\Phi"] & & 0 \arrow[r] \arrow[dd, swap, "0"]
			& \mathbb{F}^2 \arrow[r, "\text{Id}"] \arrow[dd, swap, "\text{Id}"]
			& \mathbb{F}^2 \arrow[r, "\text{Id}"] \arrow[dd, "\text{Id}"]
			& \mathbb{F}^2 \arrow[r, "\pprojfirst"] \arrow[dd, "\text{Id}"]
			& \mathbb{F} \arrow[r] \arrow[dd, swap, "0"]
			& 0 \\
			&&&&&& \\
			W(\delta=1)\colon & 0 \arrow[r]
			& \mathbb{F} \arrow[r, swap, "\pinjfirst"]
			& \mathbb{F}^2 \arrow[r, swap, "\text{Id}"]
			& \mathbb{F}^2 \arrow[r, swap, "\text{Id}"]
			& \mathbb{F}^2 \arrow[r]
			& 0. &
		\end{tikzcd}
	\end{center}
	The~$1$-invertible morphism making an interleaving pair with~$\Phi$ is defined to be the identity when possible, which also determines the other components through the commuting squares. However, this is not the only choice of an interleaving pair. Alternatively, we could define a morphism~$\Psi$ as
	\begin{center}
		\begin{tikzcd}
			V \colon \arrow[dd, "\Psi"] & & 0 \arrow[r] \arrow[dd, swap, "0"]
			& \mathbb{F}^2 \arrow[r, "\text{Id}"] \arrow[dd, swap, "\pperm"]
			& \mathbb{F}^2 \arrow[r, "\text{Id}"] \arrow[dd, "\pperm"]
			& \mathbb{F}^2 \arrow[r, "\pprojfirst"] \arrow[dd, "\pperm"]
			& \mathbb{F} \arrow[r] \arrow[dd, swap, "0"]
			& 0 \\
			&&&&&& \\
			W(\delta=1)\colon & 0 \arrow[r]
			& \mathbb{F} \arrow[r, swap, "\pinjfirst"]
			& \mathbb{F}^2 \arrow[r, swap, "\text{Id}"]
			& \mathbb{F}^2 \arrow[r, swap, "\text{Id}"]
			& \mathbb{F}^2 \arrow[r]
			& 0, &
		\end{tikzcd}
	\end{center}
	and the morphism making its interleaving pair as the exchange matrix whenever possible, which again defines the other components through commuting squares. No matter which definition we choose, the barcode of the image is~$\{ [0,2], [0,2] \}$ in both cases. The {BL-induced} matchings of the two morphisms, computed as
	\begin{align*}
		\chi_{i_\Phi}^{BL} &= \left\{
		\begin{matrix}
			\langle \cdot, 2 \rangle_{\Barc{\text{im}}} & & \langle \cdot, 2 \rangle_{\Barc{W(\delta=1)}} \\
			[0, 2] & \mapstodot & [-1, 2] \\
			[0, 2] & \mapstodot & [0 , 2]
		\end{matrix} \right\} = \left\{
		\begin{matrix}
			\langle \cdot, 2 \rangle_{\Barc{\text{im}}} & & \langle \cdot, 3 \rangle_{\Barc{W}} \\
			[0, 2] & \mapstodot & [0, 3] \\
			[0, 2] & \mapstodot & [1 , 3]
		\end{matrix} \right\} = \chi_{i_\Psi}^{BL}, \\
		\chi_{i_\Phi}^{BL} &= \left\{
		\begin{matrix}
			\langle 0, \cdot \rangle_{\Barc{\text{im}}} & & \langle 0, \cdot
			\rangle_{\Barc{V}} \\
			[0, 2] & \mapstodot & [0, 3] \\
			[0, 2] & \mapstodot & [0, 2]
		\end{matrix} \right\} = \chi_{i_\Psi}^{BL}, \\
		\chi^{BL}_\Phi &= \{([0,3], [0,3]), ([0,2], [1, 3])\}= \chi^{BL}_{\Psi},
	\end{align*}
	are therefore the same and do not respect the mapping of the morphism fully.
	
	This is not true for the matchings induced by the ladder decompositions of the interleavings. Notice that in the chosen barcode bases~$\Phi$ and~$\Psi$ are already in matching form, namely
	\begin{align*}
		\begin{aligned}
			\Phi = 
			\begin{bmatrix}
				1 & 0 \\
				0 & 1
			\end{bmatrix}
		\end{aligned}
		\qquad \text{ and } \qquad
		\begin{aligned}
			\Psi = 
			\begin{bmatrix}
				0 & 1 \\
				1 & 0
			\end{bmatrix}.
		\end{aligned}
	\end{align*}
	The matchings induced by their ladder decompositions are
	\begin{align*}
		\chi_\Phi &=  \{([0,3], [0,3]), ([0,2], [1, 3])\} \text{ and } \\
		\chi_{\Psi} &=  \{([0,3], [1,3]), ([0,2], [0, 3])\},
	\end{align*}
	which are clearly different. Despite a~$\delta$-interleaving being used in this example, the difference in the two definitions of induced matchings can be observed for a general morphism of persistence modules.
\end{example}

\subsection{Basis-Independent Partial Matchings}

As in this paper, Gonzalez Diaz and Soriano Trigueros in~\cite{basis_independent_matchings} adopt the view of morphisms as ladder persistence modules. They define a different notion of a partial matching, called \emph{basis-independent partial matching}, which is independent of the choice of order on the barcode (hence basis-independent).
\begin{definition}
	A \emph{basis-independent partial matching} between persistence modules~$V$ and~$W$, indexed over posets~$P_V$ and~$P_W$ respectively, is a function
	\begin{align*}
		\mathcal{M}_V^W\colon \overline{P_V} \times \overline{P_W} \to \Z_{\geq 0},
	\end{align*}
	where~$\overline{P_V}$ and~$\overline{P_W}$ are the sets of intervals in~$P_V$ and~$P_W$ respectively. Further, it must satisfy
	\begin{align*}
		&\sum_{1 \leq d \leq n} \sum_{1 \leq c \leq d} \mathcal{M}_V^W(a, b, c, d) \leq \mu^V([a, b]) \text{ and }\\
		&\sum_{1 \leq b \leq n} \sum_{1 \leq a \leq b} \mathcal{M}_V^W(a, b, c, d) \leq \mu^W([c, d]), 
	\end{align*}
	where~$\mu^V([a, b])$ is the multiplicity of bar~$[a, b]$ in~$\Barc{V}$ and~$\mu^W([c, d])$ is the multiplicity of bar~$[c, d]$ in~$\Barc{W}$.
\end{definition}
The ladder decomposition induced partial matchings of~\cite{jacquard2021space}, and therefore the ones we study in this paper, are examples of basis-independent partial matchings. To see this, define~$\mathcal{M}_V^W$ for a morphisms~$\Phi$ as
\begin{align*}
 	\mathcal{M}_V^W(a, b, c, d) = r_{[c, d]}^{[a, b]},
\end{align*}
where~$r_{[c, d]}^{[a, b]}$ is the multiplicity of~$R_{[c, d]}^{[a, b]}$ appearing in the ladder decomposition of~$\Phi$. Then~$\mathcal{M}_V^W(a, b, c, d)$ is the multiplicity of the pair~$([a, b], [c, d])$ in the ladder decomposition induced partial matching~$\chi_\Phi$. It is rather obvious that the sum~$\sum_{J} r_{[c, d]}^{J}$ is not bigger than the multiplicity of~$[c, d]$ in~$\Barc{W}$ and the sum~$\sum_{J} r_{J}^{[a, b]}$ is not bigger than the multiplicity of~$[a, b]$ in~$\Barc{V}$.

\subsection{\texorpdfstring{$q$}\ -Coarse Induced Partial Matchings}

Whenever the interleaving parameter~$\delta$ is too big to apply \cref{cor:cost_of_induced_matching}, we might still leverage the results of \cref{sec:quasi-barcode_form} to define partial matchings of potentially higher cost. The following result is obtained by combining \cref{prop:q-decompositions,cor:correspondence_of_matchings,cor:same_entries_qquasi,cor:matching_induced_by_decomp}.

\begin{corollary} \label{cor:q-induced_matching}
	Let~$(\Phi, \Psi)$ be a~$\delta$-interleaving pair between modules~$V$ and~$W$, and suppose there exists a parameter~$q$ such that
	\begin{align*}
		\delta < \half \min(\Xi(V_{\geq q}), \Xi(W_{\geq q})) - \qhalf.
	\end{align*}
	Then the~$(\delta + \qhalf)$-interleaving pair~$(\Phi', \Psi')$ it induces by \cref{prop:q-decompositions}~\eqref{interleaving1},~\eqref{interleaving2} or~\eqref{interleaving3} further induces a partial matching
	\begin{align*}
		&\chi_{\Phi'} \colon \Barc{V} \todot \Barc{W} \\
		&\chi_{\Phi'} = \{ ((J_1, J_2), r_{J_1}^{J_2}) \mid \mathbf{R}_{J_1 }^{J_2} \text{ appears in ladder decomposition of }\Phi' \}
	\end{align*}
	which leaves bars in~$\Barc{V_{<q}}$ and~$\Barc{W_{<q}}$ unmatched. It is of cost smaller or equal to~$\delta + \qhalf$. By \cref{cor:same_entries_qquasi} for any pair of bars~$J_V \in \Barc{V}$ and~$J_W \in \Barc{W}$ with~$|J_V| \geq 2\delta +q$ and~$|J_W| \geq 2\delta +q$ 
	\begin{align*}
		((J_V, J_W), \mu) \in \chi_{\Phi'} \iff ((J_W, J_V), \mu) \in \chi_{\Psi'},
	\end{align*}
	where~$\mu$ is the multiplicity of this pairing in both~$\chi_{\Phi'}$ and~$\chi_{\Psi'}$.
\end{corollary}

	\def\UrlBreaks{\do\/\do-}
	\bibliographystyle{acm}
	\bibliography{bibliography}
\end{document}